\documentclass{amsart}
\usepackage{amsmath,amssymb, amsthm, tikz}
\usepackage{pgfkeys}
\usepackage[left=1.4in, right=1.4in]{geometry}
\usepackage{graphics}
\usepackage[colorlinks=true]{hyperref}
\newtheorem{theorem}{Theorem}[section]
\newtheorem{proposition}{Proposition}[section]
\newtheorem{lemma}{Lemma}[section]
\newtheorem{cor}{Corollary}[section]

\newtheorem{defi}{Definition}[section]
\newtheorem{prob}{Problem}[section]

\newtheorem*{prop}{Proposition}

\newenvironment{nota}{\medskip\noindent{\sc
Notation}.}{\goodbreak\medskip}
\newenvironment{notas}{\medskip\noindent{\sc
Notations}.}{\goodbreak\medskip}
\newenvironment{remk}{\noindent{\sc
Remark}. }{\goodbreak\vskip10pt}
\newenvironment{remks}{\noindent{\sc
Remarks}. }{\goodbreak\vskip10pt}

\newenvironment{ques}{\medskip\noindent{\sc
Question}. }{\goodbreak\medskip}

\def\cn{{\mathcal N}}
\def\cb{{\mathcal B}}

\def\cs{{\mathcal S}}
\def\ct{{\mathcal T}}

\def\ce{{\mathcal E}}

\def\cc{{\mathcal C}}

\def\ca{{\mathcal A}}
\def\ch{{\mathcal H}}
\def\ck{{\mathcal K}}
\def\cl{{\mathcal L}}

\def\cu{{\mathcal U}}
\def\cv{{\mathcal V}}

\def\R{\mathbb{R}}

\def\Z{\mathbb{Z}}
\def\N{\mathbb{N}}

\def\T{\mathbb{T}}

\def\Leb{\mathrm{Leb}}

\def\d{\delta}

\def\smallskip{\par\vspace{1mm}}
\def\medskip{\par\vspace{2mm}}
\def\bigskip{\par\vspace{3mm}}

\newcount\equationcount
\def\thenumber{0}
\def\eq#1{\global\advance\equationcount by 1
   \def\thenumber{\number\equationcount}
                        {$$#1\eqno(\thenumber)$$}}

\tikzset{
xmin/.store in=\xmin, xmin/.default=-1.5, xmin=-1.5,
xmax/.store in=\xmax, xmax/.default=7.5, xmax=7.55,
ymin/.store in=\ymin, ymin/.default=-0.75, ymin=-0.75,
ymax/.store in=\ymax, ymax/.default=3.25, ymax=3.25,
}

\begin{document}

\title[integrable Hamiltonians]{A $C^1$ Arnol'd-Liouville theorem }

\author{Marie-Claude Arnaud$^{\dag}$, Jinxin Xue}

\address{Avignon Universit\'e, Laboratoire de Math\'ematiques d'Avignon 
(EA 2151)\\ F-84018  Avignon, France } 
\email{Marie-Claude.Arnaud@univ-avignon.fr}

\address{University of Chicago, Chicago, IL, US, 60637}
\email{jxue@math.uchicago.edu}

\date{}

\keywords{($C^0$-)Poisson commutativity, Hamiltonian, Arnol'd- Liouville  theorem, foliation, Lagrangian submanifolds, 
generating  functions, symplectic homeomorphisms, complete integrability.}

\subjclass[2010]{37J50, 70H20, 53D12}

\thanks{$\dag$ member of the {\sl Institut universitaire de France.}}
\thanks{supported by ANR-12-BLAN-WKBHJ}

\begin{abstract}  In this paper, we prove a version of Arnol'd-Liouville theorem for $C^1$ commuting Hamiltonians. We show that the Lipschitz regularity of the foliation by invariant Lagrangian tori is crucial to determine the Dynamics on each Lagrangian torus and that the $C^1$ regularity of the foliation by invariant Lagrangian tori is crucial to prove the continuity of Arnol'd-Liouville coordinates. We also explore various notions of $C^0$ and Lipschitz integrability.
\end{abstract}

\maketitle
 \begin{center}
Dedicated to Jean-Christophe Yoccoz
\end{center}
\section{Introduction and Main Results.}\label{SecIntro}

This article elaborates on the following question.

\begin{ques}
{\it If a Hamiltonian system has enough commuting integrals\footnote{These notions are precisely described in the rest of the introduction. }, can we  precisely describe the Hamiltonian Dynamics, even in the case of non $C^2$ integrals?}
\end{ques}
When the integrals are $C^2$, Arnol'd-Liouville Theorem (see \cite{Arno1}) gives such a dynamical description. The proof in \cite{Arno1} relies on the Abelian group action generated by the commuting Hamiltonian flows. Unfortunately, the result is not valid for $C^1$ integrals~: in this case, there is {\it a priori}  no Hamiltonian flow that we can associate to the $C^1$ integrals, the Abelian group action does not exist and so the proof in \cite{Arno1} does not work.

Note that in some case,    a $C^0$-integrability can be shown without knowing if the integrals can be chosen smooth: the case of Tonelli Hamiltonians with no conjugate points on $\T^n\times \R^n$ (see Theorem \ref{Tnonconj}). $C^0$-integrability for such Hamiltonians  is proved in \cite{AABZ} and  some partial results concerning the Dynamics on the invariant graphs are given, but no result similar to Arnol'd-Liouville Theorem is proved. The only case where a more accurate result is obtained is when the Tonelli Hamiltonian gives rise to a Riemannian metric after Legendre transform.   
Burago \& Ivanov proved in \cite{BuIv} that a Riemannian metric with no conjugate points is  smoothly integrable, but this  is specific to the Riemannian case and cannot be adapted to the general Tonelli case.

That is why we consider in this article the case of Lipschitz and $C^1$-integrability, that are intermediary between $C^0$ and $C^2$ integrability.  For a Tonelli Hamiltonian that is $C^1$-integrable, we will prove
\begin{itemize}
\item we can define global  continuous Arnol'd-Liouville coordinates, which are defined by using a symplectic homeomorphism;
\item  the Dynamics restricted to every invariant torus is $C^1$-conjugate to a rotation;
\item we can even define a flow for the continuous Hamiltonian vectorfields that are associated to the $C^1$ integrals (see Proposition \ref{PropFlow}).
\end{itemize}
For a Tonelli Hamiltonian that is  Lipschitz integrable, we will prove that the Dynamics restricted to every invariant torus is Lipschitz conjugate to a rotation, but we will obtain no information concerning the transverse   dependence of the conjugation.

In order to state our results, let us now introduce some definitions.

\begin{defi}
A $C^2$ Hamiltonian $H:\ T^*N\to \R$ for a compact Riemannian manifold $N$ is {\em Tonelli} if the following two assumptions are satisfied. 
\begin{itemize}
\item $H$ has super-linear growth, i.e. $\frac{H(q,p)}{\|p\|}\to\infty$ as $\|p\|\to\infty$. 
\item $H$ is convex in the fiber, i.e. $\frac{\partial^2 H}{\partial p^2}(q,p)$ is positive definite for all $q,p$.  
\end{itemize}
\end{defi}
For example, a mechanical Hamiltonian $H:\ T^*\T^n\to \R$ given by $H(q,p)=\frac{1}{2}\|p\|^2+V(q)$ is Tonelli. 
\begin{notas}\begin{itemize}
\item If $H$ is  a $C^1$ Hamiltonian defined on a symplectic manifold $(M^{(2n)}, \omega)$, we denote by $X_H$ the Hamiltonian vectorfield, that is defined by
$$\forall\ x\in M,\ \forall\ v\in T_xM,\quad \omega (X_H(x), v)=dH(x)\cdot v.$$
If moreover $H$ is $C^2$, the Hamiltonian flow associated to $H$, that is the flow of $X_H$, is denoted by $(\varphi_t^H)$.
\item If $H$ and $K$ are two $C^1$ Hamiltonians that are defined on $M$, their Poisson bracket is
$$\{H, K\}(x)=DH(x)\cdot X_K(x)=\omega (X_H(x), X_K(x)).$$
\end{itemize}
\end{notas}

\begin{defi}\label{DglobalC1integ}
Let $\cu\subset M$ be an open subset and let $H:M\rightarrow \R$ be a $C^2$ Hamiltonian. Then $H$ is \emph{$C^k$ completely integrable} for $k\geq 1$  in $\cu$ if 
\begin{itemize}
\item $\cu$ is invariant by the Hamiltonian flow of $H$;
\item there exist $n$ $C^k$ functions $H_1,\ H_2, \dots , H_n:\cu\rightarrow \R$ so that
\begin{itemize}
\item for every $x\in \cu$,  the maps $t\mapsto H_i\circ (\varphi_t^H(x))$ are constant ;
\item at every $x\in \cu$, the family $dH_1(x), \dots , dH_n(x)$ is independent;
\item for every $i, j$, we have $\{ H_i, H_j\}=0$ and $\{ H_i, H\}=0$. 
\end{itemize}
\end{itemize}
\end{defi}
\begin{remks}\begin{enumerate}
\item We cannot always take  $H_1=H$. At the critical points of $H$,   $dH(x)=0$   and a Tonelli Hamiltonian has always critical points. However, if we consider only the part of phase space without critical points, we can indeed take $H=H_1$.
\item Observe that when $k=1$, the $C^1$ Hamiltonians $H_1, \dots, H_n$ don't necessarily define a flow because the corresponding vector field is just continuous.   Hence the proof of Arnol'd-Liouville theorem (see for example \cite{Arno1}) cannot be used to determine what  the Dynamics is on the invariant Lagrangian submanifold $\{ H_1=c_1, \dots , H_n=c_n\}$. That is why the results we give in Theorem \ref{ThmGC1} and \ref{thciham} below are non-trivial.\\
In fact, in the setting of next definition for $k=1$ and when the Hamiltonian is Tonelli, we will prove in Proposition \ref{PropFlow} a posteriori that each $H_i$  surprisingly defines a flow.
\end{enumerate}
\end{remks}
\begin{defi}
Let   $H: U\rightarrow \R$ be a $C^2$ Hamiltonian and let $\ct$ be an invariant $C^k, ( k\geq 1$) Lagrangian torus contained in $U$. We say that $H$  is {\em locally $C^k$ completely integrable} at $\ct$ if there exists a neighborhood $\cu\subset U$ of $\ct$ such that
\begin{itemize}
\item $H$ is $C^k$ completely integrable in $\cu$;
\item $\ct$ is one leaf  of the foliation given by level sets of the $n$ integrals.
\end{itemize}
\end{defi}

We will sometimes need  the following narrower definition of $C^k$ integrability.

\begin{defi}\label{GC1int}
A $C^2$ Hamiltonian $H:\cv\rightarrow \R$ that is defined on  some open subset $\cv\subset T^*N$ is called {\em $G$-$C^k$ completely integrable} on some open subset $\cu\subset \cv$,
\begin{itemize}
\item
 if it is $C^k$ completely integrable and
 \item  if, with the same notations as in the definition of $C^k$ complete integrability, every Lagrangian submanifold $\{ H_1=c_1, \dots , H_n=c_n\}$ is the graph of a $C^k$ map.
 \end{itemize}
\end{defi}
\begin{remks}\begin{enumerate}
\item Observe that  for any $k\geq 1$, $G$-$C^k$   integrability is equivalent to the existence of an invariant $C^k$ foliation into Lagrangian graphs. The direct implication is a consequence of the fact that the $H_i$ are commuting in the Poisson sense (i.e. $\{ H_i, H_j\}=0$). For the reverse implication, denote the invariant $C^k$ foliation by $(\eta_a)_{a\in U}$ where $a$ is in some open subset of $\R^n$. Then we define a $C^k$ map $A=(A_1, \dots, A_n)$ by 
$$A(q, p)=a\Longleftrightarrow p=\eta_a(q).$$
Observe that each Lagrangian graph $\ct_a$ of $\eta_a$ is in the energy level $\{ A_i=a_i\}$. Hence $X_{A_i}(q, p)\in T_{(q, p)}\ct_{A(q,p)}$ and thus $$\{ A_i, A_j\}(q, p)=\omega(X_{A_i}(q, p), X_{A_j}(q, p))=0$$ because all the $\ct_a$ are Lagrangian. In a similar way, $\{ H, A_i\}=0$.
\item When $k\geq 2$, $\cu=T^*N$ and $H$ is a Tonelli Hamiltonian,  $C^k$-integrability implies $G$-$C^k$ integrability and that $M=\T^n$. Let us give briefly the arguments:  in this case the set of fixed points of the Hamiltonian flow $\cn=\{ \frac{\partial H}{\partial p}=0\}$ is an invariant submanifold that is a graph. Hence $\cn$ is one of the invariant tori given by Arnol'd-Liouville theorem. Therefore $M=\T^n$ and all the invariant tori of the foliation are also graphs : this is true for those that are close to $\cn$, and in this case there is a uniform Lipschitz constant because the Hamiltonian is Tonelli. Using this Lipschitz constant,  we can extend the neighborhood where they are graphs to the whole $T^*\T^n$.
\end{enumerate}
\end{remks}
It was a fundamental result of Gromov \cite{G} and Eliashberg \cite{E} in symplectic geometry that the group of symplectomorphisms on a symplectic manifold is $C^0$-closed in the group of diffeomorphisms. 

\begin{defi}\label{DefSympeo}
Following \cite{OhMu}, we call a homeomorphism a {\em symplectic homeomorphism} if its restriction to every relatively compact open subset is a uniform limit for the $C^0$-topology of a sequence of symplectic $C^\infty$ diffeomorphisms.\end{defi}
\begin{theorem}\label{ThmGC1}
Suppose that $H:\ T^*\T^n\rightarrow \R$ is a Tonelli Hamiltonian that is $G$-$C^1$ completely integrable in some open set $\cu \subset T^*\T^n$. Then there exists a neighborhood $U$ of $0$ in $\R^n$ and a symplectic homeomorphism $\phi:\ \T^n\times U\to \cu$ that is $C^1$ in the direction of $\T^n$ such that 
\begin{itemize}
\item $\forall\ c\in U$, $\phi^{-1}\circ \varphi^H_t\circ \phi(\T^n\times \{ c\})=\T^n\times \{ c\}$;
\item $\forall\ c\in U$, $\phi^{-1}\circ \varphi^H_t\circ \phi_{|\T^n\times \{ c\}}=R_{t\rho(c)}$;
\end{itemize}
where $\rho: U\rightarrow \R^n$ is a homeomorphism onto $\rho (U)$, and $R_{t\rho(c)}:\ \T^n\to\T^n$ is given by $R_{t\rho(c)}(x):=x+t\rho(c)$ mod $\Z^n$.
\end{theorem}
This gives some symplectic Arnol'd-Liouville coordinates in the $C^0$ sense (see Chapter 10 of \cite{Arno1}) and describes precisely the Dynamics on the leaves of the foliation. \\

\begin{remk}
Observe that the conjugacy $\phi$ that we obtain has the same regularity as the foliation in the direction of the leaves but is just $C^0$ in the transverse direction. If we replace the $C^1$-integrability by a Lipschitz integrability, we loose any transverse regularity and we just obtain some results along the leaves. Let us explain this.
\end{remk}

\begin{defi}
A $C^2$ Hamiltonian $H:\cv\rightarrow \R$ that is defined on  some open subset $\cv\subset T^*N$ is called {\em $G$-Lipschitz completely integrable} on some open subset $\cu\subset \cv$  if $\cu$ admits a Lipschitz foliation by invariant Lipschitz Lagrangian graphs.
\end{defi}
Let us recall that a Lipschitz graph $\cl$ admits Lebesgue almost everywhere a tangent subspace by Rademacher theorem (see \cite{EvGa}, Theorem 2, page 81). Such a graph is Lagrangian if and only if these tangent subspaces are Lagrangian. This is equivalent to asking that $\cl$  is the graph of $c+du$ where $c$ is a closed smooth $1$-form and $u:M\rightarrow \R$ is $C^{1, 1}$.
\begin{theorem} \label{ThLip}
Suppose that  the Hamiltonian $H:\ T^*\T^n\to \R$ is Tonelli  and is $G$-Lipschitz completely integrable. Then restricted to each leaf, the Hamiltonian flow has a unique well-defined rotation vector, and is bi-Lipschitz conjugate to a translation flow by the rotation vector on $\T^n$. Moreover, all the leaves are in fact $C^1$.
\end{theorem}
\begin{remk}
Observe that we do not know if the conjugacies are $C^1$.
\end{remk}

Theorem \ref{ThmGC1} is global but ask a little more that $C^1$ integrability. If we have just $C^1$ integrability (instead of $G$-$C^1$ integrability), we obtain a local result.

\begin{defi}
Let $\ct\subset M$ be a $C^1$ Lagrangian torus in a symplectic manifold $(M^{(2n)}, \omega)$ and let $H:M\rightarrow \R$ be a $C^2$ Hamiltonian. We say that $H$ {\em has positive torsion along $\ct$ } if there exist
\begin{itemize}
\item a neighbourhood $\cu$ of $\ct$ in $M$;
\item a neigbourhood $\cv$ of the zero section in $T^*\T^n$;
\item a $C^2$ symplectic diffeomorphism $\phi:\cu\rightarrow \cv$ such that $\phi(\ct)$ is the graph of a $C^1$map and

$$\forall\ (q,p)\in \cv;\ \frac{\partial^2 (H\circ \phi^{-1})}{\partial p^2}(q, p)\quad{\rm is\ positive\ definite.}$$
\end{itemize}
\end{defi}
\begin{remk}
It is proved in \cite{Wei1}, Extension Theorem in Lecture 5, as well as the proof of Theorem in Lecture 6 (see also Theorem 3.33 of \cite{MS}), that a small neighborhood of a Lagrangian $C^k$, $k\ge 1,$ submanifold $\ct$ is always $C^k$ symplectomorphic to a neighborhood of the zero section in $T^*\ct$.  

So the two  important things in the definition are that 
\begin{itemize} 
\item we can choose $C^2$ coordinates even if $\ct$ is just $C^1$: it is possible just by perturbing a $C^1$ symplectic diffeomorphism into a $C^2$ one;
\item in  the new coordinates, $H$ has to be strictly convex in the fiber direction.
\end{itemize}
\end{remk}

\begin{theorem}\label{thciham}
Let $H: M\rightarrow \R$ be a $C^2$ Hamiltonian that has an invariant Lagrangian torus $\ct$. Then, if $H$ has positive torsion along $\ct$ and is locally $C^1$ completely integrable at $\ct$, there exists a neighborhood $\cu$ of $\ct$ in $M$, a open set  $U$ containing $0$ in $\R^n$ and a symplectic homeomorphism $\phi:\T^n\times U\rightarrow\cu $ that is $C^1$ in the direction of $\T^n$ and such that:
\begin{itemize}
\item $\forall\ c\in U$, $\phi^{-1}\circ \varphi^H_t\circ \phi(\T^n\times \{ c\})=\T^n\times \{ c\}$;
\item $\forall\ c\in U$, $\phi^{-1}\circ \varphi^H_t\circ \phi_{|\T^n\times \{ c\}}=R_{t\rho(c)}$;
\end{itemize}
where $\rho: U\rightarrow \R^n$ is a homeomorphism onto $\rho (U)$, and $R_{t\rho(c)}:\ \T^n\to\T^n$ is given by $R_{t\rho(c)}(x):=x+t\rho(c)$ mod $\Z^n$.

\end{theorem}

\begin{cor}\label{C1}
Any Tonelli Hamiltonian that is $G$-$C^1$ completely integrable lies in the $C^1$ closure of the set of smooth completely integrable $($in the usual Arnol'd-Liouville sense$)$ Hamiltonians. More precisely, if $H:T^*\T^n\rightarrow \R$ is $G$-$C^1$ completely integrable on some set $\{ H < K\}$ or the whole $T^*\T^n$, there exists a sequence $(H_i)$ of $C^\infty$ Hamiltonians that are $G$-$C^\infty$ completely integrable on a $\{ H<K-\varepsilon\}$ or the whole $T^*\T^n$ such that
\begin{itemize}
\item the sequence $H_i$ uniformly converges on any compact set to $H$ for the $C^1$ topology;
\item the invariant foliation for the $H_i$ uniformly converges on any compact subset to the invariant foliation for $H$ for the $C^1$ topology;
\item the Hamiltonian flows $(\varphi_t^{H_i}),\ t\in [-1,1]$ uniformly converge to $(\varphi_t^H)$ on any compact subset for the $C^0$ topology.
\end{itemize}
\end{cor}
\begin{remks}
\begin{enumerate}
\item By the continuous dependence on parameter of solutions of ODEs, the last point is a consequence of the first one : if a sequence on Lipschitz vectorfields $(X_i)$ converge in $C^0$ topology to a Lipschitz vectorfield $X$, the the sequence of flows of the $X_i$ converge in $C^0$ topology to the flow of $X$.  
\item We don't know if we can choose the $H_i$ being Tonelli, we are just able to prove that they are symplectically smoothly conjugate to some Tonelli Hamiltonians.
\end{enumerate}
\end{remks}
\begin{cor}\label{C2} Let $H: M\rightarrow \R$ be a $C^2$ Hamiltonian that has an invariant Lagrangian torus $\ct$. Then, if $H$ has positive torsion along $\ct$ and is locally $C^1$ completely integrable at $\ct$, there exists a neighborhood $\cu$ of $\ct$ in $M$ and 
  a sequence $(H_i)$ of $C^\infty$ completely integrable Hamiltonians $H_i:\cu\rightarrow \R$ such that
\begin{itemize}
\item the sequence $H_i$ uniformly converges on any compact set to $H$ for the $C^1$ topology;
\item the invariant foliation for the $H_i$ uniformly converges on any compact subset to the invariant foliation for $H$ for the $C^1$ topology;
\item the Hamiltonian flows $(\varphi_t^{H_i}),\ t\in [-1,1]$ uniformly converge to $(\varphi_t^H)$ on any compact subset for the $C^0$ topology.
\end{itemize}
\end{cor}
{\bf Structure of the proofs and comments.}
\begin{itemize}
\item In section \ref{S2}, we will use  Herman's results concerning the conjugation of torus homeomorphisms to rotations (see \cite{Her1}) and even extend some of them to describe the Dynamics on the tori that carry a minimal Dynamics (i.e. all the orbits are dense) in the case of Lipschitz or $C^1$ complete integrability;
\item in section \ref{S3}, we will introduce a new condition, called $A$-non degeneracy,  that implies the density of the union of the minimal tori; this condition is satisfied by Tonelli Hamiltonians and Hamiltonians with positive torsion. It would be nice to have non-trivial other examples of Hamiltonians that satisfy this condition;
\item in section \ref{S4}, we will prove that $C^1$ complete integrability and $A$-non degeneracy imply the existence of Arnol'd-Liouville coordinates; this proves Theorems  \ref{ThmGC1}
and \ref{thciham}; this is done by using generating functions and Hamilton-Jacobi equations; In Section \ref{S5}, we will prove the corollaries.
\end{itemize} 
 Finally we include three appendices exploring possible relaxations of the assumption of the $C^1$ integrability in the main body of the paper.
 \begin{itemize} 
\item Appendix \ref{AppA} is devoted to the study of different possible definitions of $C^0$ complete integrability.
\item Appendix \ref{AppB} recalls some  known results concerning $C^0$ completely integrable Tonelli Hamiltonians.
\item Appendix \ref{AppC} contains the proof of Theorem \ref{ThLip}.
\end{itemize}

\section{$G$-$C^1$ complete integrability determines the Dynamics on each minimal torus}\label{S2}
The goal of this section is to prove Theorem \ref{Tminimaldynamics} that tells us  that if $H$ is $G$-Lipschitz integrable:
\begin{itemize}
\item we can define on every invariant Lagrangian torus of the foliation a rotation vector;
\item  on every such torus with a rotation vector that is completely irrational, the Dynamics is bi-Lipschitz conjugate to a minimal flow of rotations with a Lipschitz constant that is  uniform.
\end{itemize}

\begin{theorem}  \label{Tminimaldynamics}
Let $\cv\subset T^*\T^n$ be an open set. Assume that $H:\cv\rightarrow \R$ is $G$-Lipschitz   completely integrable.  We denote a  Lipschitz constant of the foliation by $K$, i.e. 
$$\forall\ q\in \T^n,\ \forall\ a,a', \qquad \frac{1}{K}\| a-a'\| \leq \| \eta_a(q)-\eta_{a'}(q)\| \leq K\| a'-a\|
.$$ where we use the notations of definition \ref{DglobalC1integ} for the integrals and denote by $\mathcal T_a=\{(H_1,\ldots,H_n)=(a_1,\ldots,a_n)\}=\{ (q, \eta_a(q)); q\in\T^n\}$ the tori of the invariant foliation.Then   the flow $(\varphi_t^H)$ restricted-projected to every Lagrangian torus $\mathcal T_a$, which is denoted by $(f^a_t)$  and is defined on $\T^n$ by  $\varphi_t^H(q, \eta_a(q))=(f_t^a(q), \eta_a(f_t^a(q)))$ satisfies 
 \begin{enumerate}
\item if $(F^a_t)$ is the lift of $(f_t^a)$ to $\R^n$, then $\frac{F^a_t(x)-x}{t}$ uniformly converges with respect to $(a,x)$  to a constant $\rho(a)\in \R^n$ as $t\to\infty$. Therefore the rotation vector $\rho(a)\in \R^n$ is well-defined and continuously depends on $a$;\item if $R_{\rho(a)}$ is minimal, there
exists a homeomorphism $h_a: \T^n\rightarrow \T^n$ such that 
$$h_a\circ f^a_t=R_{t\rho(a)}\circ h_a.$$
\item  $h_a$ is $K^n$-bi-Lipschitz.
\end{enumerate}
\end{theorem} 

\begin{remk}
Observe that we don't ask in this section that the Hamiltonian is Tonelli: the results are valid even if the Hamiltonian is very degenerate. But observe that when $H$ is constant (the very degenerate case), then there is no torus where the Dynamics is minimal and so in this case Theorem \ref{Tminimaldynamics} is almost empty. This theorem will be useful when we are sure that such tori exist.
\end{remk}

Using the following proposition, we could easily deduce a  analogue of Theorem \ref{Tminimaldynamics}  in a local $C^1$-integrable  setting.
\begin{proposition}\label{Plocalmin}
Let $H: M\rightarrow \R$ be a $C^2$ Hamiltonian that has an invariant Lagrangian torus $\ct$. Then, if $H$  is locally $C^1$ integrable at $\ct$, there exists an invariant open  neighborhood $\cu$ of $\ct$ in $M$, an open subset $U$ of $\R^n$ and a symplectic $C^2$ diffeomorphism $\psi:\T^n\times U\rightarrow \cu$ such that $H\circ \psi$ is $G$-$C^1$ integrable.
\end{proposition}

In the remaining part of this section, we will give and prove three propositions and a corollary that will imply Theorem \ref{Tminimaldynamics}. Then  we will prove Proposition \ref{Plocalmin}.

\begin{notas}
For a vector $v\in \R^n$, we use $\|v\|$ to denote its Euclidean norm. For a matrix $M\in \R^{n\times n}$, we define its norm by 
$$\|M\|=\sup_{\|v\|=1}\|Mv\|,$$
and its conorm by $m(M):=\inf_{\|v\|=1}\|Mv\|$.  \\
For a vector $\alpha\in \R^n$, we denote $R_\alpha$ the rigid rotation by $\alpha$ on $\T^n$, i.e. $$R_\alpha:\ \T^n\to\T^n,\quad R_\alpha(x)=x+\alpha,\ \mathrm{mod\ } \Z^n.$$

For a point $(q,p)\in \T^n\times\R^n$, we introduce two projections $\pi_1(q,p)=q\in \T^n$ and $\pi_2(q,p)=p\in \R^n.$
\end{notas}
\subsection{Uniform rate for the flow on the invariant tori}\label{ss21}
\begin{proposition}\label{LmUniformLipLip}
Assume that $H$ is $G$-Lipschitz completely integrable on some open subset $\cu\subset T^*\T^n$. Then  there exists a constant $K>0$ such that, restricted on each Lagrangian torus $\ct_a=\{(H_1,\ldots,H_n)=(a_1,\ldots,a_n)\}$, the flow $(f^a_t)$ satisfies the following Lipschitz estimate at Lebesgue almost every point $q\in \ct_a$
\begin{equation}\label{EqUniformLip}
\forall\ v\in \R^n,\ \forall\ t\in \R,\ \quad  \frac{\| v\|}{K}\leq \| Df_t^a(q)v\|\leq K\| v\|.
\end{equation} 
\end{proposition}
\begin{proof} 
We denote by $\eta_a: \T^n\rightarrow \R^n$ the map such that the graph of $\eta_a$ is the invariant submanifold $\ct_a$. 
Then $N:(q, a)\mapsto (q, \eta_a(q))$ is a bi-Lipschitz-homeomorphism, because $(\eta_a)$ defines a Lipschitz foliation.

{ Because of Rademacher Theorem, the set $D(N)$ where $N$ is differentiable has full Lebesgue measure. Moreover, $D(N)$ is invariant by $(\varphi_t^H)$ because the foliation is invariant.
Using the notation of Theorem \ref{Tminimaldynamics} we have $$\varphi_t^H(q,\eta_a(q))=\left(f^a_t(q), \eta_a(f^a_t(q))\right).$$
Differentiating this equation, we obtain for every $(q, a)\in D(N)$:
\begin{equation}\label{ED1}
D\varphi_t^H(q, \eta_a(q))\begin{pmatrix} 0\\ \frac{\partial\eta_a}{\partial a}(q)\end{pmatrix}=\frac{\partial f_t^a}{\partial a}(q).\begin{pmatrix} 1\\ \frac{\partial\eta_a}{\partial q}(f_t^a(q))\end{pmatrix}+\begin{pmatrix} 0\\ \frac{\partial\eta_a}{\partial a}(f_t^a(q))\end{pmatrix}\
\end{equation}
\begin{equation}\label{ED2}
D\varphi_t^H(q, \eta_a(q))\begin{pmatrix} 1\\ \frac{\partial\eta_a}{\partial q}(q)\end{pmatrix}=\frac{\partial f_t^a}{\partial q}(q).\begin{pmatrix} 1\\ \frac{\partial\eta_a}{\partial q}(f_t^a(q))\end{pmatrix}
\end{equation}
Then we use along every orbit in $D(N)$ a symplectic change of bases with matrix in the usual coordinates  at $(a, \eta_a(q))$:
$$P_{(q, a)}=\begin{pmatrix}1&0\\ \frac{\partial \eta_a}{\partial q}(q)&1\end{pmatrix}$$
We deduce from (\ref{ED1}) and (\ref{ED2}) that the matrix of $D\varphi_t^H(q, \eta_a(q))$ in the  new coordinates is
\begin{equation} \begin{pmatrix} \frac{\partial f^a_t}{\partial q}(q)&b_t(q, a)\\ 0 & d_t(q, a)\end{pmatrix}
\end{equation}\label{Emat}
where $b_t(q, a)\frac{\partial \eta_a}{\partial a}(q)=\frac{\partial f_t^a}{\partial a}(q)$ and $d_t(q, a)\frac{\partial \eta_a}{\partial a}(q)=\frac{\partial\eta_a}{\partial a}(f_t^a(q))$.

Because the foliation is biLipschitz, there exists a constant $K$ such that, for every $(q, a)\in D(N)$, we have
$$\left\|\frac{\partial \eta_a}{\partial a}(q)\right\|\leq \sqrt K\quad{\rm and} \quad \left\|\left(\frac{\partial \eta_a}{\partial a}(q)\right)^{-1}\right\|\leq \sqrt K.$$
As $d_t(q,a)=\frac{\partial \eta_a}{\partial a}(f_t^a(q))\left(\frac{\partial \eta_a}{\partial a}(q)\right)^{-1}$, we deduce that
$$\| d_t(q,a)\|\leq K\quad{\rm and}\quad\| d_t(q,a)^{-1}\|\leq K.$$
the matrix in (\ref{Emat}) being symplectic, we deduce that
$$\left\| \frac{\partial f_t^a}{\partial q}(q)\right\|=\left\| d_t(q, a)^{-t}\right\|\leq K$$
and
$$\left\| \left(\frac{\partial f_t^a}{\partial q}(q)\right)^{-1}\right\|=\left\| d_t(q, a)^{t}\right\|\leq K$$
}
 The proof is not completely finished because the result that we obtain is valid only for $(a,q)\in D(N)$. As  $D(N)$ has total Lebesgue measure, by Fubini theorem, the set of $a$ for which $N$ is differentiable Lebesgue almost everywhere in $\ct_a$ has full Lebesgue measure in $\R^n$. Hence, we obtain local Lipschitz estimates along a dense set of graphs. In other words, there exists a dense set $\cb$ of parameters $a$ such that 
$$\forall\ t\in\R,\   \forall\ a\in \cb,\ \forall\ q, q'\in\T^n,\quad  \frac{1}{K}d(q, q')\leq d(f^a_t(q), f^a_t(q'))\leq Kd(q, q').$$
Approximating any parameter by a sequence of parameters in $\cb$ and taking the limit, we obtain the same estimates for any parameter and then the wanted result (\ref{EqUniformLip}) by differentiation.
\end{proof}
\begin{cor}\label{LmUniformLip}
Assume that $H$ is $G$-$C^1$ completely integrable on some open subset $\cu\subset T^*\T^n$. Then for every compact subset $\ck\subset \cu$, there exists a constant $K>0$ such that, restricted on each Lagrangian torus $\ct_a=\{(H_1,\ldots,H_n)=(a_1,\ldots,a_n)\}$ with $\ct_a\cap \ck\not=\emptyset$, the flow $f^a_t$ satisfies the following Lipschitz estimate 
$$(\ref{EqUniformLip})\quad
\forall\ v\in \R^n,\ \forall\ t\in \R,\ \forall q\in \mathcal{T}_a \quad  \frac{\| v\|}{K}\leq \| Df_t^a(q)v\|\leq K\| v\|.
$$
\end{cor}
\begin{proof}
We  explain how to deduce Corollary \ref{LmUniformLip} from Proposition \ref{LmUniformLipLip}. We assume that $H$ is $G$-$C^1$-integrable on some open subset $\cu\subset T^*\T^n$ and that $\ck\subset \cu$ is compact and connected. Observe that $K_0=\{ a; \ct_a\cap\ck\not=\emptyset\}$ is compact. Hence the map $(a, q)\in K_0\times \T^n\mapsto (a, \eta_a(q))$ restricted to $K_0$, which is $C^1$, is bi-Lipschitz when restricted to the compact $K_0\times \T^n$ and then Corollary \ref{LmUniformLip} can be deduced from Proposition \ref{LmUniformLipLip}. 
\end{proof}
\subsection{Uniform rate and completely irrational rotation vector imply uniform Lipschitz conjugation to a flow of rotations}  \label{ss22}

\hglue 0.1cm

\noindent Next, we restrict our attention to one Lagrangian torus $\mathcal T_a$ and the associated restricted-projected flow $f^a_t$. Let  $(F^a_t)$ be the flow that is a lift of $(f^a_t)$ to $\R^n$. 
\begin{proposition}\label{LmLipConj}
Assume equation \eqref{EqUniformLip} for the flow $f^a_t$. Then 
\begin{enumerate}
\item[(A)]  The { family} $\frac{F^a_t(q)-q}{t}$ uniformly converges with respect to $(a, q)$ to a constant $\rho(a)\in \R^n$ { when $t\rightarrow +\infty$}. Therefore the rotation vector $\rho(a)\in \R^n$ is well-defined and continuously depends on $a$;
\item[(B)] if $R_{\rho(a)}:\ \T^n\to \T^n$ is minimal, there
exists a homeomorphism $h_a: \T^n\rightarrow \T^n$ such that 
$$h_a\circ f^a_t=R_{t\rho(a)}\circ h_a.$$
\item[(C)] The conjugacy $h_a$ is $K$-Lipschitz. 
\end{enumerate}
 \end{proposition}
\begin{proof}
We first prove part (A). The proof is similar to Proposition 3.1, ch 13 in \cite{Her1} with small modifications adapted to the flow instead of the map.   

We consider the following family of functions $\{F^a_t(q)-q-F^a_t(0)\ |\ t\in \R\}.$ It is known that this family is equicontinuous by Inequality \eqref{EqUniformLip}. Next, For each $t$, the function $F^a_t(q)-q-F^a_t(0)$ is zero at $q=0$, and is $\Z^n$-periodic. Again by \eqref{EqUniformLip}, using $F^a_t(q)-F^a_t(0)=\int_0^1 DF_t^a(sq)q\,ds $, we get that 
\begin{equation}\label{EqUniform}
 \sup_t\max_q\|F^a_t(q)-q-F^a_t(0)\|\leq K+1.
\end{equation}
We next pick any invariant Borel probability measure $\mu$ of $f^a_n$ and pick a $\mu$-generic point $q^*$ for which the Birkhoff ergodic theorem holds, we get 
$$ \frac{F^a_n(q^*)-q^*}{n}=\frac{\sum_{i=0}^{n-1}(F^a_{i+1}(q^*)-F^a_{i}(q^*))}{n}\to \int_{\T^N}(F_1^a(q)-q)\, d\mu(q)$$
Next for $t\in \R$, we have 
$$\frac{F^a_t(q^*)-q^*}{t}=\left(\frac{F^a_{[t]}(q^*)-q^*}{[t]}+\frac{(F^a_{t-[t]}-{\rm Id}_{\R^n})\circ F^a_{[t]}(q^*)}{[t]}\right)\frac{[t]}{t}.$$
This shows that $$\frac{F^a_t(q^*)-q^*}{t}\to  \int_{\T^N}(F_1^a(q)-q)\, d\mu(q)\ \mathrm{as\ } t\to \infty,$$ since $(F^a_{t-[t]}(q)-q)$ is uniformly bounded for all $q,t$. 
By \eqref{EqUniform}, we get that  the convergence $$\frac{F^a_t(q)-q}{t}\to \rho(a)= \int_{\T^N}(F_1^a(q)-q)\, d\mu(q)$$ is uniform on $\T^n$. Using Equation (\ref{EqUniform}) and Proposition XIII 1.6. page177 of \cite{Her1}, we deduce also
\begin{equation}\label{EqUniformnombrota}
 \sup_t\max_q\|F^a_t(q)-q-t\rho(a))\|\leq 2(K+1).
\end{equation}
This implies that the convergence to $\rho(a)$ is uniform in $(q, a)$.

We next work on part (B). We assume the rotation vector $\rho(a)$ is completely irrational and apply Proposition 3.1, page 181, in ch 13 of \cite{Her1} to the  map $f^a_1:\ \T^n\to\T^n$ to get that there is a homeomorphism $h_a:\ \T^n\to\T^n$ such that $h_a\circ f_1^a(q)= h_a(q)+\rho(a)$. Iterating this formula, we get $h_a\circ f^a_n(q)= h_a(q)+n\rho(a)$ for all $n\in \Z$. It remains to show that $h_a\circ f^a_t(q)= h_a(q)+t\rho(a)$ for all $t\in \R$. 

Equation (\ref{EqUniformnombrota}) tells us that 
 $$|F^a_t(q)-q-t\rho(a)|\leq 2(K+1)$$
 uniformly for all $q$ and $t$. 
In the expression $$\frac{1}{t}\left(\int_{0}^t (F^a_s(q)-s\rho(a))\,ds\right),$$
 we have just established the uniform boundedness for all $t\in \R\setminus\{0\}$ and $q$ in a fundamental domain, and the equicontinuity follows from \eqref{EqUniformLip}. By Arzela-Ascoli, we can extract uniformly convergent subsequence as $t_k\to \infty$. Denote by $g(q)$ the limit 
 $$g(q)=\lim_k \frac{1}{t_k}\left(\int_{0}^{t_k} (F^a_s(q)-s\rho(a))\,ds\right).$$
 Next, we consider 
\begin{equation}
\begin{aligned}
 g(F^a_t(q))&=\lim_k \frac{1}{t_k}\left(\int_{0}^{t_k} (F^a_s(F_t^a(q))-s\rho(a))\,ds\right)\\
 &=\lim_k \frac{1}{t_k}\left(\int_{0}^{t_k} (F^a_{s+t}(q)-(s+t)\rho(a))\,ds\right)+t\rho(a)\\
 &=\lim_k \frac{1}{t_k}\left(\int_{0}^{t_k} (F^a_{s}(q)-s\rho(a))\,ds\right)+\lim_k \frac{1}{t_k}\left(\left(\int_{t_k}^{t_k+t}-\int_{0}^{t} \right)(F^a_{s}(q)-s\rho(a))\,ds\right) +t\rho(a)\\
 &=g(q)+t\rho(a).
\end{aligned}
\end{equation}
 Choosing $t=1$, we get $g(F^a_1(q))=g(q)+\rho(a)$. Then taking difference with the equation $h_a(F^a_1(q))=h_a(q)+\rho(a)$, we get $g(F^a_1(q))-h_a(F^a_1(q))=g(q)-h_a(q)$. By assumption we have that $R_{\rho(a)}$ is minimal, so is $f^a_1$, therefore we get that $g(q)-h_a(q)=$const. This shows that $g$ is a homeomorphism as $h_a$ is.

Moreover, Proposition 3.1, page 181, in ch 13 of \cite{Her1} implies that we can choose for $h_a$:
$$h_a(q)-q=\lim_{k\rightarrow +\infty}\sum_{j=0}^{k-1}\frac{F^a_j(q)-j\rho(a)}{k}.$$
Hence $h_a$ is $K$-Lipschitz because all the $F^a_j$ are $K$-Lipschitz. \end{proof}

\subsection{Lipschitz conjugation to a completely irrational rotation implies bi-Lipschitz conjugation}

\hglue 1truecm

\noindent The main goal of this section is to prove: 

\begin{proposition}\label{thlipconj}
Let $g: \T^n\to\T^n$ be a bi-Lipschitz homeomorphism such that the family $(g^k)_{k\in\Z}$ is $K$-equi-Lipschitz. If the rotation vector $\alpha$ of $g$ is such that $R_\alpha$ is ergodic, then $g$ is bi-Lipschitz-conjugated to $R_\alpha$ by some conjugation $h$ in $\mathrm{Homeo}(\T^n)$.   Moreover, if $h\circ g=R_\alpha\circ h$, then the conjugation $h$ is $K$-Lipschitz and its inverse $h^{-1}$ is $K^n$-Lipschitz. 
\end{proposition} 
\begin{remks}
\begin{itemize}
\item Observe that in the proof of  Proposition 3.2, page 182, in ch 13 of \cite{Her1}, M.~Herman raised the question of the Lipschitzian property of such $h^{-1}$ and that we give here a positive answer.
\item Observe too that with the notations of section \ref{ss22}, Proposition \ref{thlipconj} implies that every $h_a^{-1}$ is $K^n$-Lipschitz.
\end{itemize}
\end{remks}

\begin{proof}

We know from Proposition 3.2, page 182, in ch 13 of \cite{Her1}  that if the rotation vector $\alpha$ of $g$ is such that $R_\alpha$ is minimal, and if \begin{equation}\label{Econj}h\circ g=R_\alpha\circ h,\end{equation} then $h$ is $K$-Lipschitz. Let us prove that $h^{-1}$ is $K^n$-Lipschitz.

\begin{lemma}\label{Lmeas}
Let $g:\T^n\rightarrow \T^n$ be a  bi-Lipschitz homeomorphism so that  $\sup_{k\in\Z}\| Dg^k\|_{C^0}=K<+\infty$. Then $g$ has an invariant Borel probability measure $\mu$ so that for all Borel subset $A\subset \T^n$, we have
$$\frac{1}{K^n}\Leb(A)\leq \mu (A)\leq K^n\Leb(A).$$
In particular, $\mu$ is equivalent to $\mathrm{Leb}$.
\end{lemma}
\begin{proof}
We apply Krylov--Bogolyubov process (see \cite{KrBo}) to the Lebesgue measure.   In other words, for every $k\geq 1$, we define
$$\mu_k=\frac{1}{k}\sum_{j=0}^{k-1} g_*^k\Leb.$$
Observe that for every open set $A$, we have
$$g^k_*\Leb(A)=\int_Ad(g^k_*\Leb)=\int_A |{\rm det} Dg^k|d\Leb\in \left[\frac{1}{K^n}\Leb(A), K^n\Leb(A)\right].$$
We deduce that $\mu_k(A)\in \left[\frac{1}{K^n}\Leb(A), K^n\Leb(A)\right]$. Let now $\mu$ be a limit point of the sequence $(\mu_k)$. Then  $\mu$ is invariant by $g$ and satisfies
$$\frac{1}{K^n}\Leb(A)\leq \mu(A)\leq K^n\Leb(A).$$
\end{proof}
By Rademacher theorem, $g$ is Lebesgue almost everywhere differentiable. Because $g$ is bi-Lipschitz,  we deduce that the set $Z$ of the points $q\in \T^n$ such that $g$ is differentiable at every $g^k(q)$ for $k\in\Z$ has full Lebesgue measure. We will use the following lemma that is easy to prove by using the definition of the differential and the Lipschitz property.\\
\begin{lemma}\label{Lipinv}
Let $G:\T^n\rightarrow \T^n$ a bi-Lipschitz homeomorphism. Assume that $G$ is differentiable at some $q_0\in \T^n$. Then $G^{-1}$ is differentiable at $G(q_0)$ and 
$$DG^{-1}(G(q_0))=\left( DG(q_0)\right)^{-1}.$$
\end{lemma}
By Rademacher theorem, $h$ is Lebesgue almost everywhere differentiable and we have $\| Dh\|\leq K$. \\
As we have $h=R_\alpha\circ h\circ g^{-1}$ and Lemma \ref{Lipinv}, the set $E$ of the points of $Z$ where $h$ is differentiable is invariant under $g$. Moreover, we know that $\Leb(E)=1$, i.e. $\Leb (\T^n\backslash E)=0$. By Lemma \ref{Lmeas}, this implies that $\mu(\T^n\backslash E)=0$ and then $\mu(E)=1$. Moreover, we have on $E$:
$$ (Dh)\circ g=Dh\cdot (Dg)^{-1}.$$
Let us now prove that for $\mu$-almost every $x\in E$ and every unit vector $v\in T_x\T^n$, we have $\|Dh(x)v\|\geq \frac{1}{K^n}\| v\|$. If not, there exist $\varepsilon >0$ and a subset $E'\subset E$ with non-zero $\mu$-measure  such that for every $x\in E'$, there is a unit vector $v\in T_{x}\T^n$ such that $\|Dh(x)v\|\leq\frac{1-\varepsilon}{K^n}\| v\|<\frac{1}{K^n}\| v\|$. Because of Equation (\ref{Econj}),  we have for every $k\in \N$  
$$Dh(g^{k}x)\cdot Dg^{k}(x)=Dh(x).$$
Hence if we denote by $w_k\in T_{g^kx}\T^n$ the unit vector $w_k=\frac{1}{\| Dg^k(x)v\|}Dg^k(x)v$, then we deduce from $\| (Dg^{k})^{-1}\|\leq {K}$ that
$$\|Dh(g^{k}x)w_k\|=\frac{1}{\| Dg^k(x)v\|}\|Dh(x)v\| \leq(1-\varepsilon)\frac{\|w_k\|}{K^n\| Dg^k(x)v\|}\leq (1-\varepsilon)\frac{\| w_k\|}{K^{n-1}}.
$$
If now we compute $|{\rm det}(Dh(g^kx))|$ in some orthonormal basis with the first vector equal to $w_k$, we obtain because $\| Dh\| \leq K$ that  
$$|{\rm det}(Dh(g^kx))|\leq \frac{1-\varepsilon}{K^{n-1}}.K^{n-1}=1-\varepsilon.$$
Hence we have proved that $|{\rm det} (Dh)|\leq 1-\varepsilon$ on the set
$$\ce=\bigcup_{k\in \N} g^k(E').$$
As $\mu$ is ergodic and $\mu(E')\not=0$, we have $\mu(\ce)=1$ and then $\Leb (\ce)=1$. Integrating and using a change of variables, we finally obtain
$$1=\int_{\T^n}{ d\Leb}=\int_{\T^n}|{\rm det}Dh(x)|\,d\Leb(x)\leq 1-\varepsilon.$$
Hence we have proved that for Lebesgue almost every $x\in \T^n$, we have $\| Dh^{-1}(x)\|\leq {K^n}.$\\
This implies in particular that $h^{-1}$ is $K^n$-Lipschitz.\end{proof}

This ends the proof of Theorem \ref{Tminimaldynamics}.

\subsection{Proof of Proposition \ref{Plocalmin}}
We assume that $H$  is locally $C^1$ completely integrable at $\ct$. Hence there exists an invariant open neighborhood $\cu$ of $\ct$ and  $n$ $C^1$ functions $H_1,\ H_2, \dots , H_n:\cu\rightarrow \R$ so that
\begin{itemize}
\item the $H_i$ are constant on $\ct$;
\item at every $x\in \cu$, the family $dH_1(x), \dots , dH_n(x)$ is independent;
\item for every $i, j$, we have $\{ H_i, H_j\}=0$ and $\{ H_i, H\}=0$;
\item $\ct=\{ H_1=a_1, \dots, H_n=a_n\}$.
\end{itemize}
Then if we take a smaller $\cu$,  $\ch= (H_1, \dots , H_n):\cu\rightarrow \R^n$ is a $C^1$ submersion that gives an invariant by $(\varphi_t^H)$  Lagrangian foliation and (see \cite{Wei1}, Extension Theorem in Lecture 5, as well as the proof of Theorem in Lecture 6)
 there exists a $C^1$ symplectic embedding $\phi:\cu\rightarrow \T^n\times \R^n$ that maps $\ct$ onto $\T^n\times \{ 0\}$ and then  the foliation onto a foliation into $C^1$ graphs. Perturbing $\phi$, we can assume that $\phi$ is smooth and symplectic and that the foliation is into graphs (but of course $\phi(\ct)$, which is a graph, cannot be the zero section). The diffeomorphism $\psi$ of Proposition \ref{Plocalmin} is the   given by $\psi=\phi^{-1}$.
\qed 

 \section{The A-non-degeneracy condition}\label{S3}
 In section \ref{S2}, we succeeded in describing the Dynamics on the invariant tori having  a completely irrational rotation vector. If we want to extend this to all the tori, we need to know that the union of these tori is dense in the set where the Hamiltonian is completely integrable.  That is why we introduce a new notion of non-degeneracy that implies this density. We will prove in Proposition \ref{PconvexAnondeg} that when the Hamiltonian is Tonelli or with positive torsion, this condition is always satisfied.
 \begin{defi}\label{Defsfc}
 Assume that $H$ is $G$-$C^1$ completely integrable and denote by $$(\ct_a)_{a\in U}=(\{ (q, \eta_a(q))\ :\quad q\in\T^n\} )$$ the invariant foliation. 
 
 The function $\mathsf c:U\rightarrow \R^n$ is defined by $\mathsf c(a)=\int_{\T^n}\eta_a(q)dq$.
 \end{defi}
 We will prove the following proposition in Section \ref{ssconjC1}.
 \begin{proposition}\label{LmCohom}Assume that $(q, a)\in \T^n\times U\mapsto (q, \eta_a(q))\in\cu$  is a $C^1$ foliation of an open subset $\cu$ of $\T^n\times\R^n$ into Lagrangian graphs.
The function $\mathsf c:U\rightarrow \R^n$ is a $C^1$ diffeomorphism from $U$ onto its image.
\end{proposition}
  \begin{defi}
 Assume that $H$ is $G$-$C^1$ completely integrable with invariant foliation $$(\ct_a)_{a\in U}=(\{ (q, \eta_a(q)), q\in\T^n\} ).$$
 We define the {\em $A$-function} $A: \mathsf c(U)\rightarrow \R$ by $A(c)=H(0, \eta_{\mathsf c^{-1}(c)}(0))$.
 Observe that we have
 $$\forall\ q\in \T^n,\ \forall\ c\in \mathsf c(U),\quad H(q, \eta_{\mathsf c^{-1}(c)}(q))=A(c).$$
 \end{defi}
 \begin{remks}
 \begin{enumerate}
 \item When $H$ is a  Tonelli Hamiltonian,  the $A$ function is exactly the $\alpha$-function of Mather (see \cite{M}).
 \item Observe that the $A$-function is $C^1$.
 \end{enumerate}
 \end{remks}
 \begin{defi}
 Let $H:\cu\rightarrow \R$ be a $G$-$C^1$ completely integrable Hamiltonian with the invariant foliation described in the above definition. We say that  $H$ is {\em $A$-non-degenerate} if for every non-empty open subset $V\subset \mathsf c(U)$,  $\nabla A (V)$  is contained in no resonant hyperplane of $\R^{n}$
 $$k_1x_1+\dots k_nx_n=0$$
 with $(k_1, \dots, k_n)\in \Z^n\backslash \{ (0, \dots, 0)\}$ .
 \end{defi}
 We will use later the following result.
 \begin{proposition}\label{Pnondegdense}
  Assume that $H$ is $G$-$C^1$ completely integrable and $A$-non-degenerate and denote by $(\ct_a)_{a\in U}=(\{ (q, \eta_a(q)), q\in\T^n\} )$ the invariant foliation. Then for every non-empty open subset $V\subset \mathsf c(U)$, the set $\nabla A(V)$ contains a completely irrational $\alpha=(\alpha_1, \dots, \alpha_n)$ i.e. $\alpha$ that doesn't belong to any resonant hyperplane.
 \end{proposition}
 \begin{proof}
 Assume that $V\subset \mathsf c(U)$ is a  non-empty open subset of $\mathsf c(U)$ such that  $\nabla A(V)$ contains no completely irrational number. Then $V$ is contained in the countable union of the backward images by the continuous map $\nabla A$ of all the  closed resonant hyperplanes. As $V$ is Baire, there exists a non-empty open subset $W$ of $V$ such that $\nabla A(W)$ is contained in some resonant hyperplane. This contradicts the definition of $A$-non-degeneracy.
 
  \end{proof}
  
  \begin{proposition}\label{PconvexAnondeg}
   Assume that $H$ is $G$-$C^1$ completely integrable and   strictly convex in the momentum variable, i.e.
 
    $$\forall\ (q, p)\not=\ (q, p')\in \cu,\ \forall\ \lambda \in (0, 1),\quad (q, \lambda p+ (1-\lambda)p')\in \cu $$
   $$\Rightarrow \\H(q, \lambda p+ (1-\lambda)p')<\lambda H(q, p)+ (1-\lambda)H(q,p').$$
   Then $H$ is $A$-non-degenerate and even $\nabla A$ is a homeomorphism onto $\nabla A(\mathsf c(U))$.
  \end{proposition} 
  \begin{proof}
  We recall that if $F:V\rightarrow \R$ is a  $C^1$ function that is strictly convex, then $DF$ is an homeomorphism from $V$ onto its image (see theorem 1.4.5. in \cite{F}).
 
  Hence we just have to prove that $A$ is strictly convex. Let us fix two distinct $c_1$, $c_1$ in $\mathsf c(U)$ and $\lambda\in (0, 1)$. \\
Because $H$ is strictly convex in the fiber direction,  we have for every $q\in \T^n$, 
  $$H(q, \lambda\eta_{\mathsf c^{-1}(c_1)}(q)+(1-\lambda)\eta_{\mathsf c^{-1}(c_2)}(q))<\lambda H(q, \eta_{\mathsf c^{-1}(c_1)}(q))+(1-\lambda)H(q, \eta_{\mathsf c^{-1}(c_2)}(q)).$$
  As $H(q, \eta_{\mathsf c^{-1}(c)}(q))=A(c)$, we have
\begin{equation}\label{Econv}
H(q, \lambda\eta_{\mathsf c^{-1}(c_1)}(q)+(1-\lambda)\eta_{\mathsf c^{-1}(c_2)}(q))<\lambda A(c_1)+(1-\lambda)A(c_2).\end{equation}
 Observe that $\lambda c_1+(1-\lambda )c_2$ is equal to 
  $$\int_{\T^n}(\lambda\eta_{\mathsf c^{-1}(c_1)}(q)+(1-\lambda)\eta_{\mathsf c^{-1}(c_2)}(q))dq=\int_{\T^n}\eta_{\mathsf c^{-1}(\lambda c_1+(1-\lambda)c_2)}(q)dq.
  $$
  Hence the 1-form $\lambda\eta_{\mathsf c^{-1}(c_1)}+(1-\lambda)\eta_{\mathsf c^{-1}(c_2)}-\eta_{\mathsf c^{-1}(\lambda c_1+(1-\lambda)c_2)}$ is exact and so there exists a $q_0\in\T^n$ so that
  $$\lambda\eta_{\mathsf c^{-1}(c_1)}(q_0)+(1-\lambda)\eta_{\mathsf c^{-1}(c_2)}(q_0)=\eta_{\mathsf c^{-1}(\lambda c_1+(1-\lambda)c_2)}(q_0).$$
  Replacing in Equation (\ref{Econv}), we obtain
  $$H(q_0, \eta_{\mathsf c^{-1}(\lambda c_1+(1-\lambda)c_2)}(q_0))<\lambda A(c_1)+(1-\lambda)A(c_2),$$
  i.e.
  $$A(\lambda c_1+(1-\lambda)c_2)<\lambda A(c_1)+(1-\lambda)A(c_2).$$
  \end{proof}
 
 \section{The symplectic homeomorphism in the case of $C^1$ complete integrability}\label{S4}
\noindent The goal of this section is to prove the following theorem, which, joint with Proposition \ref{PconvexAnondeg},  implies Theorem \ref{ThmGC1} and Theorem \ref{thciham}.

\begin{theorem}\label{TAnondeghomeo}
Suppose $H:\cu\subset \ T^*\T^n\rightarrow \R$ is a   $G$-$C^1$ completely integrable Hamiltonian that is $A$-non-degenerate.Then there exist a neighborhood $U$ of $0$ in $\R^n$ and a symplectic homeomorphism $\phi:\ \T^n\times U\to \cu$ that is $C^1$ in the direction of $\T^n$ such that 
\begin{itemize}
\item $\forall\ c\in U$, $\phi^{-1}\circ \varphi^H_t\circ \phi(\T^n\times \{ c\})=\T^n\times \{ c\}$;
\item $\forall\ c\in U$, $\phi^{-1}\circ \varphi^H_t\circ \phi_{|\T^n\times \{ c\}}=R_{t\rho(c)}$;
\end{itemize}
where $\rho: U\rightarrow \R^n$ is a homeomorphism onto $\rho (U)$.
\end{theorem}

 \subsection{ A generating function for the Lagrangian foliation}

\begin{proposition}\label{Pregfol}
Assume that $(q, a)\in \T^n\times U\rightarrow (q, \eta_a(q))\in\cu$  is a $C^1$ foliation of an open subset $\cu$ of $\T^n\times\R^n$ into Lagrangian graphs.  Then there exists a $C^1$ map $S:\cu\rightarrow \R$  such that
\begin{itemize}
\item $\frac{\partial S}{\partial q}(q, a)=\eta_a(q)-\mathsf c(a)$;
\item $\frac{\partial S}{\partial q}$ is $C^1$ in the two variables $(q, a)$;
\item $\frac{\partial S}{\partial a}$ is $C^1$ in the variable $q$.
\end{itemize}
\end{proposition}

\begin{proof}
 As the graph of $\eta_a$ is Lagrangian, there exist a $C^2$ function $u_a: \T^n\rightarrow \R$ such that
$$\eta_a=\mathsf c(a)+\frac{\partial u_a}{\partial q}.$$
Observe that $u_a$ is unique up to the addition of some constant.  We then define:
$$S(q, a)=u_a(q)-u_a(0).$$

Then $S$ has a derivative with respect to $q$ that is $\frac{\partial S}{\partial q}(q, a)=\eta_a(q)-\mathsf c(a)=\frac{\partial u_a}{\partial q}(q)$, which is $C^1$ in the two variables $(q,a)$.

Moreover, 
$$S(q, a)=\int_0^1\frac{\partial S}{\partial q} (\gamma_q(t),a)\dot\gamma_q(t)dt=\int_0^1\left( \eta(q)-\mathsf  c(a)\right)\dot\gamma_q(t)dt$$
 where $\gamma_q:[0, 1]\rightarrow \T^n$ is any $C^1$ arc joining $0$ to $q$. This implies that $S$ is $C^1$ in the variable $a$ and that
\begin{equation}\label{Epartial}\frac{\partial S}{\partial a}(q, a)=\int_0^1\frac{\partial^2 S}{\partial a\partial q} (\gamma_q(t),a)\dot\gamma_q(t)dt= \int_0^1\left( \frac{\partial\eta_a}{\partial a} (\gamma_q(t),a)-\frac{\partial \mathsf c}{\partial a}(a)\right)\dot\gamma_q(t)dt.
\end{equation}
We deduce that $S$ has partial derivatives with respect to $a$ and $q$ that continuously depend on $(q, a)$. This implies that $S$ is in fact $C^1$.

Moreover, we deduce from Equation (\ref{Epartial}) that for every $\nu\in\R$ and $\d q\in \R^n$, we have
$$\frac{\partial S}{\partial a}(q+\nu\d q, a)-\frac{\partial S}{\partial a}(q, a)=\left(\int_0^\nu \frac{\partial \eta_a}{\partial a}(q+t\delta q)\delta q dt\right)-\nu \frac{\partial \mathsf c}{\partial a}(a)\delta q.
$$
This implies that $\frac{\partial S}{\partial a}$ has a partial derivative with respect to $q$  in the direction $\d q$ that is:
$$\frac{\partial^2 S}{\partial q\partial a}(q, a)\cdot\delta q=\left(\frac{\partial \eta_a}{\partial a}(q)-\frac{\partial \mathsf c}{\partial a}(a)\right)\delta q.
$$
Because the foliation is $C^1$, then all these partial derivatives continuously depend on $q$. This implies that $\frac{\partial S}{\partial a}$ is $C^1$ in the $q$ variable.
\end{proof}
\subsection{The $C^1$ property of the conjugation}\label{ssconjC1}
Now we will prove Proposition \ref{LmCohom} that we rewrite by using the generating function.
\begin{prop} {\rm\bf \ref{LmCohom}} Assume that $(q, a)\in \T^n\times U\mapsto (q, \mathsf c(a)+\frac{\partial S}{\partial q}(q, a))\in\cu$  is a $C^1$ foliation of an open subset $\cu$ of $\T^n\times\R^n$ into Lagrangian graphs.
The function $\mathsf c:U\rightarrow \R^n$ is a $C^1$ diffeomorphism from $U$ onto its image.
\end{prop}
\begin{proof}
Observe that $\mathsf c$ is injective.  If $a\not=a'$, then $\ct_a\cap\ct_{a'}=\emptyset$ because we have a foliation.  Assume that $\mathsf c(a)=\mathsf c(a')$. Let $q_0$ be a critical point of $S(\cdot , a)-S(\cdot ,a')$. Then we have 
$\frac{\partial S}{\partial q}(q_0, a)-\frac{\partial S}{\partial q}(q_0, a')=0$ and  
$$\mathsf c(a)+\frac{\partial S}{\partial q}(q_0, a)=\mathsf c(a')+\frac{\partial S}{\partial q}(q_0, a'),$$
which contradicts that $\ct_a\cap\ct_{a'}=\emptyset$.

Hence $\mathsf c$ is $C^1$ and injective. To prove the lemma, we just have to prove that $D\mathsf c(a)$ is invertible at every point of $U$.\\
If not, let us choose $a\in U$ and a non-zero vector $v\in \R^n$ such that $D\mathsf c(a)v=0$. Now, we choose $q_0\in \T^n$ such that $q\mapsto \frac{\partial S}{\partial a}(q, a)v$ attains its maximum at $q_0$. Then we have
$$\frac{\partial \eta_a}{\partial a}(q_0)v=D\mathsf c(a)v+ \frac{\partial^2 S}{\partial q\partial a}(q_0, a)v=0.$$
This contradicts the fact that $(\eta_a)$ defines a $C^1$-foliation.
\end{proof}
Let us recall the notation that we introduced in the statement of Theorem \ref{Tminimaldynamics}:

\begin{nota}  The flow $(\varphi_t^H)$ restricted-projected to every Lagrangian torus $\mathcal T_a$, which is denoted by $(f^a_t)$, is defined on $\T^n$ by  $\varphi_t^H(q, \eta_a(q))=(f_t^a(q), \eta_a(f_t^a(q)))$.
\end{nota}

\begin{theorem}\label{ThmWKAMDc}
Suppose $H:\cu\subset T^*\T^n\rightarrow \R$ is a   $G$-$C^1$ completely integrable Hamiltonian that is $A$-non-degenerate and let  $(q, a)\in \T^n\times U\rightarrow (q, \mathsf c(a)+\frac{\partial S}{\partial q}(q, a))\in\cu$  be the  invariant $C^1$ foliation.   Then   $$g_a(q):=q+D\mathsf c(a)^{-1}\frac{\partial S}{\partial a}(q, a)$$
is a $C^1$ diffeomorphism conjugating  $(f_t^a)$ to $R_{t\rho(a)}$ where $\rho$  was defined in Proposition \ref{LmLipConj} and in fact $\rho=(\nabla A)\circ\mathsf c$.\\
Hence, when $\nabla A(\mathsf c(a))$ is completely irrational, we have $g_a=h_a$
up to an additive constant where $h_a$ is the conjugacy associated to the torus $\mathcal{T}_a$ given by Theorem \ref{Tminimaldynamics}.
\end{theorem}
\begin{proof}
Write $\eta_a(q)=\mathsf c(a)+\frac{\partial S}{\partial q}(q, a)$. We have
\begin{equation}\label{E1}
H\left(q, \mathsf c(a)+\frac{\partial S}{\partial q}(q, a)\right)=A(\mathsf c(a))
\end{equation}
where $\mathsf c(a)=\int_{\T^n}\eta_a(q)dq$ $C^1$ depends on $a$ and $A$ is the $A$-function.  Differentiating Equation (\ref{E1}) with respect to $a$, we deduce  that if $t\mapsto (q(t), p(t)) =\varphi_t^H(q,p)$ is contained in $\ct_a$, we have
$$D\mathsf c(a)\cdot \dot q(t)+\frac{\partial^2 S}{\partial q\partial a}(q(t), a)\dot q(t)=D\mathsf c(a)\cdot \nabla  A(\mathsf c(a))
$$
and then  
$$D\mathsf c(a)\left(q(t)-q(0)-t \nabla  A(\mathsf c(a))\right)+\frac{\partial S}{\partial a}(q(t), a)-\frac{\partial S}{\partial a}(q(0), a)=0.$$

 We deduce
$$f^a_t(q)+D\mathsf c(a)^{-1}\frac{\partial S}{\partial a}(f^a_t(q), a)=q+D\mathsf c(a)^{-1}\frac{\partial S}{\partial a}(q, a)+t \nabla  A(\mathsf c(a)).$$
If we define $g_a(q)=q+D\mathsf c(a)^{-1}\frac{\partial S}{\partial a}(q, a)$, we have proved that
$$g_a(f_t^a(q))=g_a(q)+t\nabla  A(\mathsf c(a)).$$
Hence $g_a$ is a semi-conjugation between $(f_t^a)$ and $(R_{t\nabla  A(\mathsf c(a))})$. 
\begin{lemma}\label{Lvectrota}
Let $f: \T^n\rightarrow \T^n$ be a homeomorphism. Assume that $g, h\in C^0(\T^n, \T^n)$ are homotopic to identity and such that $g\circ f=g+\alpha$  and $h\circ f=h+\beta$ for some $\alpha, \beta\in\T^n$. Then $\beta=\alpha$.
\end{lemma}
\begin{proof}
From  $g\circ f=g+\alpha$ and  $h\circ f=h+\beta$, we deduce that 
\begin{equation}\label{Ediffcoh}(g-h)\circ f=g-h+\alpha-\beta.\end{equation}
Because $g$ and $h$ are homotopic to identity, there exists a continuous map $G:\T^n\rightarrow \R^n$ such that the class of $G(q)$ modulo $\Z^n$ is   $(g-h)(q)$.  Then Equation (\ref{Ediffcoh}) and the continuity of $G$ implies that $G\circ f-G$ is a constant $B\in\R^n$ whose class modulo $\Z^n$ is  $\alpha-\beta$. 
If $\mu$ is a Borel probability measure that is invariant by $f$, we then obtain by integrating $$\int G\,d\mu=\int G \,d\mu +B.$$
Hence $B=0$ and   $\alpha=\beta$ mod $\Z^n$.
\end{proof}

We denote by $\ca$ the set of $a\in U$ such that $\rho(a)$ is completely irrational. We deduce from Lemma \ref{Lvectrota} that for $a\in \ca$, we have $\rho(a)=\nabla  A(\mathsf c(a))$ and then $g_a(f_t^a(q))=g_a(q)+t\rho(a)$.
We deduce from $g_a(f_t^a(q))=g_a(q)+t\rho(a)$ and $h_a(f_t^a(q))=h_a(q)+t\rho(a)$ that $(g_a-h_a)\circ f_t^a=(g_a-h_a)$. Hence $g_a-h_a$ is constant Lebesgue almost everywhere and then by continuity constant. Hence $g_a$ itself is a conjugation between $(f_t^a)$ and $(R_{t\rho (a)})$.

As $H$ is $G$-$C^1$ completely integrable and $A$-non-degenerate, we deduce from Proposition \ref{Pnondegdense} that $\ca$ is dense in $U$. If now $a\in U$, let $\ck$ be a compact neighborhood of $\ct_a$ in $\cu$.  We deduce from Theorem \ref{Tminimaldynamics} that   there exists a constant $K$ such that for every $a\in\ca$, if $\ct_a\cap\ck\not=\emptyset$, then $h_a$ is $K$-bi-Lipschitz. \\
Then we choose a sequence $(a_i)$ in $\ca$ such that
\begin{itemize}
\item every $\ct_{a_i}$ meets $\ck$;
\item $(a_i)$ converges to $a$.
\end{itemize}
Every $g_{a_n}$ is a $K$-bi-Lipschitz homeomorphism and the sequence $(g_{a_n})$ $C^0$ converges to $g_a$. We deduce that $g_a$ is a $K$-bi-Lipschitz homeomorphism  too and then a $C^0$ conjugation between $(f^a_t)$ and $(R_{t\nabla A(\mathsf c(a))})$. This implies that $\nabla A(\mathsf c(a))=\rho(a)$.

Moreover, all the $g_a$s are $C^1$. Observe that a $C^1$ homeomorphism that is bi-Lipschitz is a $C^1$ diffeomorphism. Hence all the $g_a$s are $C^1$ diffeomorphisms.
\end{proof}
\begin{proposition}\label{PropFlow}
Each continuous vectorfield $X_{H_i}$ generates a flow.
\end{proposition}
\begin{proof}
Note that if we define for every   $a\in U$ and every $i\in\{ 1, \dots, n\}$ $\gamma_i^a$ by $$\gamma_i^a=Dg_a(q)\frac{\partial H_i}{\partial p}(q, \eta_a (q)),$$ then the flow defined on $\T^n\times U$ by
$$\varphi_t^{H_i}(q, \eta_a(q))=(g_a^{-1}\circ R_{t\gamma_i^a}\circ g_a(q), \eta_a(g_a^{-1}\circ R_{t\gamma_i^a}\circ g_a(q)))$$
preserves the foliation in $\eta_a$, is continuous and $C^1$ along the $\ct_a$ and such that
$$\frac{\partial \varphi_t^{H_i}}{\partial t}_{|t=0}(x)=X_{H_i}(x).$$
\end{proof}
\begin{remk}
We don't affirm that the vectorfields $X_{H_i}$ are uniquely integrable.
\end{remk}

\subsection{ The symplectic homeomorphism}\label{SSHomeo}
With the notations used in the previous part, if we define
$$\phi_0(q,a)=(g_a^{-1}(q), \eta_a(g_a^{-1}(q))),$$
then $\phi_0$ satisfies all the conclusions of Theorem  \ref{TAnondeghomeo} except the fact that it is not symplectic.

Let $\cc: \T^n\times U\rightarrow\T^n\times \mathsf c(U)$ be the $C^1$ diffeomorphism that is defined by
$\cc (q, a)=(q, \mathsf c(a))$. Then we define the homeomorphism $\phi: \T^n\times \mathsf c(U)\rightarrow \cu$ by
$$\phi=\phi_0\circ \cc^{-1}.$$
Then  $\phi$ is a homeomorphism that is $C^1$ in the direction of $\T^n$ and such that:
\begin{itemize}
\item $\forall c\in \mathsf c(U)$, $\phi^{-1}\circ \varphi^H_t\circ \phi(\T^n\times \{ c\})=\T^n\times \{ c\}$;
\item $\forall c\in \mathsf c(U)$, $\phi^{-1}\circ \varphi^H_t\circ \phi_{|\T^n\times \{ c\}}=R_{t\rho\circ \mathsf c^{-1}(c)}$.
\end{itemize}
To see that $\phi$ is a homeomorphism,  we note that $\phi=f_0\circ \psi_0^{-1}$ is a composition of a $C^1$ diffeomorphism given by the foliation $$f_0:(q, c)\mapsto (q, c+\frac{\partial S}{\partial q}(q, \mathsf c^{-1}(c))),$$ and the inverse function of the conjugation preserving the  standard foliation (by the $\T^n\times \{ c\}$)  $$\psi_0:(q, c)\mapsto (g_{\mathsf c^{-1}(c)}(q), c).$$ Note that $\psi_0$ is $C^1$ in the $q$ direction but not $C^1$ in the $c$ direction. 

Observe that if  $S(q, a)$ is the function that we introduced in Proposition \ref{Pregfol}, the function $(q, c)\mapsto \cs(q,c)=S(q, \mathsf c^{-1}(c))$ is a generating function for $\phi$ in the sense that if $\phi(x, c)=(q, p)$, then we have
$$\begin{cases}
& p=c+\frac{\partial \cs }{\partial q}(q,c),\\
& x=q+\frac{\partial \cs}{\partial c}(q,c).\\
\end{cases}$$

Now we prove that $\phi$ can be approximated in $C^0$ topology by smooth symplectomorphisms $\phi_\epsilon$ as $\epsilon\to 0$. We introduce $$\mathcal S_\epsilon(q,c)=\cs(q,c)*\varphi_\epsilon(c)*\varphi_\epsilon(q)$$ where $\varphi_\epsilon$ is a $C^\infty$ approximating Dirac-$\delta$ function as $\epsilon\to 0$ and $*$ is the convolution. We get that $\mathcal S_\epsilon(q,c)$ is $C^\infty$ in both variables and is a well-defined function on $T^*\T^n$ ($\Z^n$ periodic in $q$). Moreover, when $\epsilon$ tends to $0$, $(\mathcal S_\epsilon)$ tends to $\cs$ in $C^1$ topology uniformly on compact subsets and $(\frac{\partial \mathcal S_\epsilon}{\partial q})$ tends to $\frac{\partial \cs}{\partial q}$ in $C^1$ topology uniformly on compact subsets.

Hence, if $V$ is relatively compact, for  $\epsilon$ small enough, $f_\epsilon (q, p) =(q, c+\frac{\partial \mathcal S_\epsilon}{\partial q} (q,c))$ is a smooth map that is $C^1$ close to the diffeomorphism that gives the initial foliation $(q, c)\mapsto (q, c+\frac{\partial \cs}{\partial q} (q,c))$ and thus is a smooth diffeomorphism when restricted to $V$. Observe too that the approximating foliations are smooth and $C^1$ converge to the initial one.

Let $W$ be a compact subset of $\T^n\times \R^n$. We now define $\psi_\epsilon: (q,c)\mapsto (q+\frac{\partial \mathcal S_\epsilon}{\partial c}(q, c), c)$.  Let us recall that for every $c$, the map $g_{\mathsf c^{-1}(c)}: q\mapsto q+\frac{\partial \cs}{\partial c}(q, c)$ is a $C^1$ diffeomorphism of $\T^n$. Moreover, the maps $q\mapsto q+ \frac{\partial \cs_\epsilon}{\partial c}(q, c)$ and $q\mapsto {\bf 1}+ \frac{\partial^2 \cs_\epsilon}{\partial q \partial c}(q, c)$ uniformly converge on $W$ to $q\mapsto q+ \frac{\partial \cs}{\partial c}(q, c)$ and $q\mapsto {\bf 1}+ \frac{\partial^2 \cs}{\partial q \partial c}(q, c)$ respectively. We deduce that the restriction of $\psi_\epsilon$ to $W$ is also a smooth diffeomorphism for $\epsilon$ small enough.

If $W$ is relatively compact, we can now define $\phi_\epsilon$ by $\phi_\epsilon= f_\epsilon\circ \psi_\epsilon^{-1}$ on $W$ for $\epsilon$ small enough. Then $\mathcal S_\epsilon(q,c)$ is a generating function of the smooth symplectic diffeomorphim $\phi_\epsilon(x,c)=(q,p)$ i.e.
$$\begin{cases}
& p=c+\frac{\partial \mathcal S_\epsilon}{\partial q}(q,c),\\
& x=q+\frac{\partial \mathcal S_\epsilon}{\partial c}(q,c).\\
\end{cases}$$
Hence $\phi_\epsilon$ is symplectic and $\phi_\epsilon$ tends to $\phi= f_0\circ \psi_0^{-1}$ in $C^0$ topology on every compact set when $\varepsilon$ tends to $0$.
 \qed
\subsection{The smooth approximation}\label{S5}
The goal of this section is to prove Corollaries \ref{C1} and \ref{C2}. We will use some notations from Section \ref{SSHomeo}.

The Hamilton-Jacobi equation $H(q,c+\frac{\partial \cs}{\partial q}(q,c))=A(c)$ can be rewritten as 
$$H(q,p)=A\circ \pi_2\circ f_0^{-1}(q,p)$$
where $f_0: (q, c)\mapsto (q, c+\frac{\partial \cs}{\partial q} (q,c))$ is the $C^1$ map that defines the foliation.

The corollaries then follow by approximating $A$ by a sequence of $C^\infty$ functions  $\{A_\epsilon=A*\varphi_\epsilon\}$ of only $c\in \R^n$ and approximating $f_0$ by the sequence of $C^\infty$ diffeomorphisms $\{f_\epsilon\}$  constructed in the proof of Theorems \ref{ThmGC1} and \ref{thciham}. We denote $$H_\epsilon(q,p)=A_\epsilon\circ \pi_2\circ f_\epsilon^{-1}(q,p).$$
By the construction of $A_\epsilon$, we get that $$\|A_\epsilon-A\|_{\mathcal K, C^1}\to 0$$ on every compact subset $\mathcal K$ and that every $A_\epsilon$ has positive definite Hessian.
In the proof of Theorems \ref{ThmGC1} and \ref{thciham}, we proved that on every compact subset $\mathcal K$, we have
$$\|f_\epsilon-f_0\|_{\mathcal K, C^1}\to 0.$$
This implies that
 $$\|H_\epsilon-H\|_{\mathcal K, C^1}\to 0.$$
 This proves the first and third bullet points in Corollaries \ref{C1} and \ref{C2}. The second bullet point follows from $\|f_\epsilon-f_0\|_{\mathcal K, C^1}\to 0$.

 \appendix
 
 \section{$C^0$ integrability and $C^0$ Lagrangian submanifolds}\label{AppA}

In this appendix, we study various notions of $C^0$ integrability and prove a result on the relations between two notions of $C^0$ integrability. 
\subsection{Different notions of $C^0$ integrability}
There are different notions of $C^0$ integrability existing in literature. We give a list here. 

The following definition of $C^0$ Lagrangian submanifold was first introduced in \cite{Her2} by Herman.
\begin{defi} [$C^0$ Lagrangian submanifold {\it \#1}]\label{DefLagHerman}
Suppose that a $n$-dimensional $C^0$ submanifold  of $T^*N$ is the  graph of a one-form $p(q)dq$. We say that this manifold is \emph{ $C^0$ Lagrangian in the sense}  \#1 if the one-form $p(q)\,dq$ is closed in the distribution sense.
\end{defi}
\begin{remks} \begin{enumerate}
\item It can be proved that $p(q)\,dq$ is closed in the distribution sense if and only if its integral along every closed homotopic to a point loop is zero (see Proposition \ref{Pcohom}).
\item Observe that it is proved in \cite{F} (Theorem 4.11.5) that every $C^0$-Lagrangian graph in the sense {\it\#}1 that is invariant by a Tonelli flow is in fact a {\em Lipschitz }Lagrangian graph. Hence it has a tangent space Lebesgue almost everywhere and this tangent space is Lagrangian.
\end{enumerate}
\end{remks}
\begin{proposition}\label{Pcohom} Suppose $N=\T^n$. If the graph $\mathcal T$ of $p(q)\,dq$ is $C^0$ Lagrangian in the sense  \#1, then there exists a unique $c\in \R^n$ and a unique $C^1$ function $u:\T^n\rightarrow \R$ such that $p=c+du$. \end{proposition}
Then $c$ is called the cohomology class of $\mathcal T$.

\begin{proof}
We  define the cohomology class for the torus $\mathcal T$.

 Suppose $\mathcal T=\{ (q,p (q)),\ q\in \T^n\}$.
Then we pick cycles $\gamma_j\subset \T^n$ with homology class $e_j:=(0,\ldots,0,1,0,\ldots,0)\in H_1(\T^n,\Z)$ where $1$ is in the $j$-th entry. Then the cohomology class $c\in H^1(\T^n,\R)$ of $\mathcal T$ is a vector in $\R^n$ determined by the equation 
\begin{equation}\label{EqCohom}
\langle c,e_j\rangle=\int_{\gamma_j} p(q)\,dq.\end{equation}
To see that the cohomology class is well-defined, we pick another cycle $\eta_j$ homologous to $\gamma_j$, so the difference $\gamma_j-\eta_j$ of two cycles can be realized as the boundary of a piece of surface $S_j$.  We next smoothen $p(q)$ to $p_{\epsilon}(q)$ by convoluting each of the components of $p(q)=(p_{1},\ldots,p_{n})$ by the same $C^\infty$ approximating Dirac-$\delta$ function. The components of $p_{\epsilon}(q)$ satisfy $\partial_{q_j}p_{i,\epsilon}=\partial_{q_i}p_{j,\epsilon}$ since $pdq$ is closed in the distribution sense. This implies that $p_{\epsilon}dq$ is a smooth closed one-form. 

So we get 
\begin{equation*}
\begin{aligned}
\int_{\gamma_j-\eta_j} p(q)\,dq
=& \lim_{\epsilon\to 0}\int_{\partial S_j}p_{\epsilon}(q)\,dq= \lim_{\epsilon\to 0}\int_{S_j}d(p_{\epsilon}(q)\,dq)=0.\end{aligned}
\end{equation*}
Next, we get that $(p(q)-c)\,dq$ is exact. To see this, we pick any $C^\infty$ curve $\gamma_q:\ [0,1]\to \T^n$ with $\gamma_q(0)=0$ and $\gamma_q(1)=q$ and define $u_c(q)=\int_0^1(p(\gamma_q)-c)\cdot \dot\gamma_q\,dt$. This function $u_c$ is well-defined and $C^1$ by the closedness of $pdq$ and the definition of the cohomology class $c$. 
\end{proof}

The following definition of $C^0$ Lagrangian submanifold was introduced in \cite{HLS} motivated by the $C^0$ closedness of the symplectomorphism group of Gromov-Eliashberg. 

\begin{defi}[$C^0$ Lagrangian submanifold {\it\#}2]\label{DefLagHLS}
We say a $C^0$ submanifold of a symplectic manifold is \emph{ $C^0$ Lagrangian in the sense}  \#2, if it is   symplectically homeomorphic to a smooth Lagrangian submanifold via a symplectic homeomorphism $($see Definition \ref{DefSympeo}$)$. 
\end{defi}

It is proved in Proposition 26 of \cite{HLS} that a $C^0$ Lagrangian graph  in the sense  {\it\#}1 is necessary a $C^0$ Lagrangian submanifold in the sense  {\it\#}2. However, the other direction of implication for a submanifold that is a graph is not clear. 
 
Using the above two notions of $C^0$ Lagrangian submanifold, we get two notions of $C^0$ integrability. 
\begin{defi}[$C^0$ integrability {\it\#}1]\label{DefAABZ}
We say that a Tonelli Hamiltonian $H:\ T^*\T^n\to \R$ is \emph{ $C^0$ integrable in the sense \#1} if there exists a continuous foliation of $T^*\T^n$ by invariant and $C^0$ Lagrangian graphs in the sense \#1.
\end{defi}
This definition of $C^0$ integrability was given by Arcostanzo-Arnaud-Bolle-Zavidovique in \cite{AABZ}.  Similarly, we have the following. 
\begin{defi}[$C^0$ integrability {\it\#}2]\label{DefHLS}
We say a Tonelli Hamiltonian $H:\ T^*\T^n\to \R$ is \emph{ $C^0$ integrable in the sense {\it\#}2} if there exists a continuous foliation of $T^*\T^n$ by invariant and $C^0$ Lagrangian graphs in the sense \#2. 
\end{defi}
 
Our next notion of $C^0$ integrability is based on the following remarkable $C^0$ rigidity result of Poisson bracket discovered by Cardin-Viterbo in \cite{CV}: 
if $M$ is a symplectic manifold and $F_n,G_n,F,G$ are in $C^\infty(M)$, if $(F_n, G_n)$ $C^0$ converges to $(F, G)$ and 
$$\lim_{n\rightarrow\infty}\|\{F_n,G_n\}\|_{C^0}= 0,$$
then $\{F, G\}=0$.  

With this result, it is natural to define the Poisson commutativity for $C^0$ Hamiltonians ({ this definition with a different name was  introduced in  in \cite{CV},} see also Chapter 2.1 of \cite{PR}). 
\begin{defi}[$C^0$ Poisson commutativity]\label{DefPoisson0}
For two $C^0$ functions $F$ and $G$ on a symplectic manifold $M$, we say that $F$ and $G$ \emph{Poisson commute}, i.e. $\{F,G\}=0$, if there exist two sequences $\{F_n\},\{G_n\}$ with $F_n, G_n\in C^\infty(M)$ satisfying $(F_n,G_n)\xrightarrow{C^0}(F,G)$ and
$$\lim_{n\to \infty}\|\{F_n,G_n\}\|_{C^0}= 0.$$
\end{defi}
Next, we define the third version of $C^0$-integrability, which can be considered as a $C^0$ version of Arnol'd-Liouville integrability.
\begin{defi}[$C^0$-integrability {\it\#}3] \label{DefC0}Suppose that  $\mathcal U\subset T^*\T^n$ is open.  
We say that a $C^0$ Hamiltonian $H:\ T^*\T^n\to \R$ is \emph{$C^0$ integrable} in $\mathcal U$ in the sense \#3, if there is a sequence of $C^2$ $n$-uples $$\mathsf H_i(q,p):=(H_{1,i}(q,p),\ldots,H_{n,i}(q,p)):\ \mathcal U\to \R^n,$$ satisfying
\begin{enumerate}
\item[(1)] $($Poisson commutativity$)$
$\mathsf H_{i}\to \mathsf H:=(H_1,\ldots,H_n),$ in the $C^0$ norm as $i\to \infty$ and \\
$\|\{H_{j,i}, H_{k,i}\}\|_{C^0}\to 0$ for all $j,k$, as $i\to \infty$. We assume that $H_1=H$.
\item[(2)] $($Non-degeneracy 1$)$The matrix $M_i$ formed by $\partial_p H_{j,i},\ j=1,\ldots,n$ as rows is invertible and the inverse $M_i^{-1}$ is uniformly bounded on $\mathcal U$ and in $i$; 
\item[(3)] $($Non-degeneracy 2$)$  the limiting tuple $\mathsf H(q,.):=(H_{1},\ldots,H_{n}):\ \mathcal U\to\R^n,$ for each fixed $q\in \T^n$ as a function of $p$ is an homeomorphism onto its image and this image $U$ does not depend on $q$. 
\end{enumerate}
\end{defi}
\begin{remk}
\begin{enumerate}
\item
In item (1) of Definition \ref{DefC0}, the Poisson commutativity of the Hamiltonians $H_j,\ j=1,\ldots,n$ for $n>2$ is narrower than applying Definition \ref{DefPoisson0} directly to each pair. Indeed, by Definition \ref{DefPoisson0}, the sequence of $\{H_{i,n}\}$ defining $\{H_i,H_j\}=0$ might be different from that used to defining $\{H_i,H_k\}=0$ for $ j\neq k$. We stick to the narrower definition here for simplicity. It is not clear to us how to recover Theorem \ref{ThmMainC0} below using the broader definition.
\item Since we do not need to talk about flow invariant objects in Definition \ref{DefC0}, different from Definition \ref{DefAABZ} and \ref{DefHLS}, we assume the Hamiltonian to be $C^0$ only, instead of $C^2$ as in Definition \ref{DefAABZ} and \ref{DefHLS}.
\item We need the non-degeneracy 1 hypothesis  to prove Theorem \ref{ThmMainC0}. This hypothesis tells us that { the map $\mathsf H(q,.)^{-1}$ is uniformly Lipschitz with respect to the variable $p$.}  It can be slightly weakened to allow $\|M_i^{-1}\|^2_{C^0}$ to blow up to infinity slowly as $i\to\infty$ provided the following limit holds  \begin{equation}\label{EqWeaker}\max_{j,k}\|\{H_{j,i}, H_{k,i}\}\|_{C^0}\|M_i^{-1}\|^2_{C^0}\to 0,\ \mathrm{as\ }i\to \infty.\end{equation}
\item Observe that  the non-degeneracy 1 hypothesis implies that the function $\mathsf H$ defined in point (3) of Definition   \ref{DefC0} is a local bi-Lipschitz homeomorphism for each fixed $q$, which is a part of  the non-degeneracy 2 hypothesis.
\end{enumerate}
\end{remk}

In our $C^1$ Arnol'd-Liouville theorem, we rely crucially on the $C^1$ or Lipschitz assumption to conjugate the Dynamics on each Lagrangian torus $$\mathcal T_a:=\{(q,p)\in \mathcal U\ |\ (H_1,\ldots,H_n)=a\in \R^n\}$$ to a translation on $\T^n$ by $\rho(a)$, from both the side of Herman theory and the side of generating functions. When we drop the $C^1$ or Lipschitz assumption, we lose control of the Dynamics and we do not know how to get a conjugacy. However, we still have the following theorem on the topological structure of the phase space.  

\begin{theorem}\label{ThmMainC0}
Suppose $H$ is a $C^0$-integrable Hamiltonian defined on $\mathcal U$ in the sense \#3, then the  level sets $\mathcal T_a:=\{(q,p)\in \mathcal U\ |\ (H_1,\ldots,H_n)=a\}$ for $a\in U$ define   a $C^0$ foliation of $\mathcal U$ by Lagrangian graphs in the sense \#1 $($i.e. in the sense of distribution$)$.\\
\end{theorem}

\begin{proof}
Hypothesis of non-degeneracy  2  implies that every  level set $\mathcal T_a$  for $a\in U$ is a continuous graph and that the foliation $(\mathcal T_a)$ is continuous.

Hypothesis of non-degeneracy  1 and   the usual   implicit function theorem imply that 
if  $a\in U$, then for $i$ large enough the set
$$\mathcal T_{i,a}:=\{(q,p)\in \mathcal U\ |\ (H_{1,i},\ldots,H_{n,i})=a\}$$
is the graph of some $C^2$ function $p_i:\T^n\rightarrow U$.

The $C^0$ convergence of $\mathsf H_{i}$ to $\mathsf H$ implies that $p_i$ $C^0$-converges to $p$ if $\mathcal T_a$ is the graph of $p$.

We now think each vector $\frac{\partial H_{j,i}}{\partial q}$ and $ \frac{\partial H_{j,i}}{\partial p}$ as a row vector. We compute
$$\frac{\partial H_{j,i}}{\partial q}+ \frac{\partial H_{j,i}}{\partial p}\cdot \frac{\partial p_i}{\partial q}=0.$$
This gives
$$\frac{\partial H_{j,i}}{\partial q}\cdot \left(\frac{\partial H_{k,i}}{\partial p}\right)^t+ \frac{\partial H_{j,i}}{\partial p}\cdot \frac{\partial p_i}{\partial q}\cdot\left(\frac{\partial H_{k,i}}{\partial p}\right)^t=0.$$
Permuting the subscripts $k$ and $j$, we get
$$\frac{\partial H_{k,i}}{\partial q}\cdot \left(\frac{\partial H_{j,i}}{\partial p}\right)^t+ \frac{\partial H_{k,i}}{\partial p}\cdot \frac{\partial p_i}{\partial q}\cdot\left(\frac{\partial H_{j,i}}{\partial p}\right)^t=0.$$
Taking difference, this implies that 
$$\{H_{j,i},H_{k,i}\}+\frac{\partial H_{j,i}}{\partial p}\cdot \left(\frac{\partial p_i}{\partial q}-\left(\frac{\partial p_i}{\partial q}\right)^t\right)\cdot\left(\frac{\partial H_{k,i}}{\partial p}\right)^t=0.$$
We denote by $P_i$ the matrix $(\{H_{j,i},H_{k,i}\})_{j,k}$ and $M_i$ the matrix with $\frac{\partial H_{j,i}}{\partial p},\ j=1,\ldots,n$ as rows. 
We get the following abbreviation
$$M_i\left(\frac{\partial p_i}{\partial q}-\left(\frac{\partial p_i}{\partial q}\right)^t\right)M_i^t=-P_i.$$
By assumption $M_i$ is nondegenerate and $\|M_i^{-1}\|$ is uniformly bounded in $i$ and $ P_i\to 0$ in the $C^0$ topology. So we get in the $C^0$ topology
$$\left(\frac{\partial p_i}{\partial q}-\left(\frac{\partial p_i}{\partial q}\right)^t\right)=-M_i^{-1}P_i M_i^{-t}\to 0.$$
By integrating against a $C^\infty$ test function $\psi:\ \T^n\to \R$, we see that the limiting one-form $p(q)dq$ is closed in the distribution sense. The proof also goes through if we assume \eqref{EqWeaker} instead.

 \end{proof}
 
 \begin{remk}
 If $H$ is a Tonelli Hamiltonian, Theorem \ref{ThmMainC0} tells us that $C^0$ complete integrability in the sense {\it \#}3  implies $C^0$ complete integrability in the sense  {\it\#}1. Open questions remain that we list now.
 \end{remk}
 
 \begin{prob}
Is a Tonelli Hamiltonian that is $C^0$ completely integrable in the sense  \#1 or  \#2 $C^0$ completely integrable in the sense  \#3?
\end{prob}

 \begin{prob}
Is a graph that is $C^0$ Lagrangian in the sense 
  \#2 necessarily $C^0$ Lagrangian in the sense   \#1?
\end{prob}

\begin{prob} Does there exist any Tonelli Hamiltonian that is $C^0$ integrable $($in the sense  \#1 or  \#2$)$ but not $C^1$ integrable? not Lipschitz integrable?
\end{prob}

 When $H$ is Tonelli and $C^0$ integrable, a lot of questions concerning the Dynamics restricted to the leaves remain open. We will explain in next section what is known and what is unknown.

 \section{Smooth Hamiltonian That are $C^0$ completely integrable}\label{AppB}
 The Hamiltonian that we use in Definition \ref{DefC0} and Theorem \ref{ThmMainC0} is assumed to be $C^0$ only, thus there is a priori no Dynamics.  In this section, we discuss existing dynamical results if we assume more regularity for the Hamiltonian. 
 More precisely, we assume that  the Hamiltonian $H:\ T^*\T^n\to \R$ is Tonelli and $C^0$ integrable in the sense {\it\#}1.

 The following results are proved in \cite{Arna1},\cite{Arna4} and \cite{AABZ}, . 

 \begin{theorem}\label{Tnonconj}
For a Tonelli Hamiltonian $H:\ T^*\T^n\to \R$, the following assertions are equivalent
\begin{itemize}
\item    $H$ has no conjugate points;
\item $H$ is $C^0$ integrable in the sense \#1;
\item $H$ admits a continuous foliation into Lipschitz invariant Lagrangian graphs.
\end{itemize}
 \end{theorem}
 
 In this case, a lot of leaves are in fact $C^1$.
 
 \begin{theorem}Suppose that $H:\ T^*N\to \R$ is Tonelli, where $N$ is a compact manifold, and is $C^0$ integrable in the sense  \#1. Then there exists a flow invariant $G_\delta$ subset $\mathcal G$ of $\ T^*N$ such that each leaf in $\mathcal G$ is $C^1$.
 \end{theorem}
 
 If  we ask a little more regularity of the Hamiltonian, it can be shown that the Dynamics is non-hyperbolic.
 
 \begin{theorem} Suppose that $H:\ T^*\T^n\to \R$ is a $C^3$ Tonelli Hamiltonian and  is $C^0$ integrable in the sense  \#1, then all the Lyapunov exponents of the Hamiltonian flow are zero, with respect to any invariant Borel probability measure.
 \end{theorem}
 
 If $H$ is $C^\infty$ and is $C^0$ integrable in the sense  {\it\#}1, some K.A.M. theorems can be proved close to the completely periodic tori (see \cite{AABZ}) that have a lot a nice consequences that we give now.

  \begin{theorem} Suppose that $H:\ T^*\T^n\to \R$ is a $C^\infty$ Tonelli Hamiltonian  and  is $C^0$ integrable in the sense  \#1, then
  \begin{enumerate}
  \item there exists a dense subset of $T^*\T^n$ with Lebesgue positive measure that is foliated by smooth invariant Lagrangian graphs on which the Dynamics is conjugate to a Diophantine rotation;
  \item  there exists an dense subset of $T^*\T^n$  that is foliated by smooth invariant Lagrangian graphs on which the Dynamics is conjugate to a rational rotation;
 \item there exists a dense $G_\delta$ subset of $T^*\T^n$ that is foliated by invariant Lagrangian graphs on which the Dynamics is strictly ergodic.
  \end{enumerate}
   \end{theorem}
 We recall that a dynamical system is {\em strictly ergodic} if it is uniquely ergodic and if the unique invariant Borel probability measure has full support.

The following problem was posed in \cite{AABZ} for the first case of $C^0$ integrability.
\begin{prob}\label{ProbAABZ}  Suppose $H$ is Tonelli and $C^0$ integrable  in the sense  \#1,  \#2 or  \#3. Can an invariant torus of the Hamiltonian flow carry two invariant measures with different rotation vectors?
\end{prob} \section{Smooth Hamiltonian that are Lipschitz completely integrable}\label{AppC}
 In this section, we assume further that $H$ is a G-Lipschitz completely integrable Hamiltonian.
 
We assume here more regularity, in particular in the transverse direction, than the $C^0$ integrabilities in Appendix A or B, but slightly less regularity than the $C^1$ integrability. 

We begin by proving Theorem \ref{ThLip} that we recall.

\begin{theorem} 
Suppose that  the Hamiltonian $H:\ T^*\T^n\to \R$ is Tonelli  and is $G$-Lipschitz completely integrable. Then restricted to each leaf, the  Hamiltonian flow has a unique well-defined rotation vector, and is bi-Lipschitz conjugate to a translation flow by the rotation vector. Moreover, all the leaves are in fact $C^1$.
\end{theorem}
 
\begin{proof}
First, we deduce from propositions \ref{LmUniformLipLip}, \ref{LmLipConj} and \ref{thlipconj} that   restricted to each leaf $\ct_a$, the  Hamiltonian flow has a unique well-defined rotation vector $\rho(a)$, and if the rotation vector is completely irrational, then there exists a bi-Lipschitz conjugacy conjugating the flow on the leaf  $\ct_a$ to the rigid translation by $t\rho(a)$. Observe too that $\rho(a)$ continuously depends on $a$: this is a consequence of point(A) of Proposition \ref{LmLipConj}.

Let us prove that $\rho$ is injective. We will use some results due to J.~N.~Mather concerning the minimizing measures (see \cite{Masor} and \cite{M}). We assume that $a\not=a'$ satisfy $\rho(a)=\rho(a')$. Observe that any minimizing measure with cohomology class $\mathsf c(a)$ (resp. $\mathsf c(a')$, see Definition \ref{Defsfc} for the definition of $\mathsf c$) is supported in $\ct_a$ (resp $\ct_{a'}$). Moreover, every minimizing measure supported in $\ct_a$ (resp $\ct_{a'}$) has $\rho(a)$ (resp. $\rho(a')$) for rotation vector. We deduce that every minimizing measure with rotation vector $\rho(a)=\rho(a')$ is minimizing with cohomology class $\mathsf c(a)$ and $\mathsf c(a')$, and then its support in in $\ct_a\cap \ct_{a'}=\emptyset$. We obtain then a contradiction.

Hence $a\mapsto \rho(a)$ is a continuous and injective map. By the invariance of domain theorem (see for example \cite{GH}, Theorem 18.9, page 110), $\rho$ is a homeomorphism and then for a dense set $\ca$ of $a$, $\rho(a)$ is completely irrational. Hence every $a$ can be approximated by a sequence $(a_n)$ such that every $\rho(a_n)$ is completely irrational.  Then there exists a $K$ bi-Lipschitz conjugation $h_k:\T^n\rightarrow \T^n$ such that $f^{a_k}_t= h_k^{-1}\circ R_{t\rho(a_k)}\circ h_k$. 
 By Arzela-Ascoli, we can extract from $(h_k)$ a converging subsequence to some $h:\T^n\to\T^n$ that is $K$-bi-Lipschitz and such that 
 $$f^a_t=h^{-1}\circ R_{t\alpha}\circ h$$
 for some $\alpha$. Then  necessarily $\alpha=\rho(a)$ and we obtained the wanted bi-Lipschitz conjugation. To conclude that the leaves are $C^1$, we use the following result.
 
  \begin{theorem}[Theorem 2 of \cite{Arna4}] Suppose $H:\ T^*\T^n\to \R$ is Tonelli, and that $\mathcal G$ is a $C^0$ Lagrangian graph  in the sense  \#1 which is invariant under the Hamiltonian flow. Suppose the time-1 map of the Hamiltonian flow on $\mathcal G$ is bi-Lipschitz conjugate to a rotation, then the graph  $\mathcal G$ is $C^1$.
 \end{theorem}
  \end{proof}
 So in this case, Problem \ref{ProbAABZ} is answered negatively.

\section*{Acknowledgment}
\noindent M.-C. A. is supported by the Institut Universitaire de France and  by the ANR-12-BLAN-WKBHJ\\
J.X. is supported by the NSF grant DMS-1500897.
\bibliographystyle{amsplain}

\providecommand{\bysame}{\leavevmode\hbox to3em{\hrulefill}\thinspace}
\providecommand{\MR}{\relax\ifhmode\unskip\space\fi MR }
\providecommand{\MRhref}[2]{%
  \href{http://www.ams.org/mathscinet-getitem?mr=#1}{#2}
}
\providecommand{\href}[2]{#2}
\begin{thebibliography}{}

\end{thebibliography}


\begin{thebibliography}{10}
 \bibitem{Arna1} M.-C.  Arnaud. {\it Green bundles, Lyapunov exponents and regularity along the supports of the minimizing measures.} Annales de l'Institut Henri Poincare (C) Non Linear Analysis. Vol. 29. No. 6. Elsevier Masson, 2012.

 \bibitem{Arna4} M.-C. Arnaud. {\it Fibr\'es de Green et r\'egularit\'e des graphes $C^0$-Lagrangiens invariants par un flot de Tonelli.} Annales Henri Poincar\'e. Vol. 9. No. 5. Birkh\"auser Basel, 2008.

\bibitem{AABZ} M. Arcostanzo, M.-C. Arnaud, P. Bolle and  M. Zavidovique. {\it Tonelli Hamiltonians without conjugate points and $C^0$ integrability.} Mathematische Zeitschrift 280, no. 1-2 (2015): 165-194.

\bibitem{Arno1}  Arnol'd,V.~I. {\it  Les m\'ethodes math\'ematiques de la m\'ecanique classique}. (French) Traduit du russe par Djilali Embarek. \'Editions Mir, Moscow, 1976. 470 pp.


\bibitem{BuIv} D.~Burago and S.~Ivanov.
 Riemannian tori without conjugate points are flat.
 GAFA, 4:259--269, 1994.

\bibitem{CIPP} G.~Contreras, R.~Iturriaga, G.~P.~Paternain,  \& M.~Paternain. . {\it The Palais-Smale condition and Ma\~n\'e's critical values.}   Annales Henri Poincar\'e (Vol. 1, No. 4, pp. 655-684). Birkh\"auser Verlag, 2000.

 
\bibitem{CV} F.~Cardin and C.~Viterbo. {\it Commuting Hamiltonians and Hamilton-Jacobi multi-time equations.} Duke Mathematical Journal 144.2 (2008): 235-284.

\bibitem{E} Ya M. Eliashberg. {\it A theorem on the structure of wave fronts and its applications in symplectic topology.} Functional Analysis and Its Applications 21.3 (1987): 227-232.

\bibitem{EvGa} L.~C.~Evans and R.~F.~Gariepy. {\it Measure
theory and fine properties of functions.} CRC Press,
Boca Raton, FL, 1992

\bibitem{F} A.~Fathi. {\it Weak KAM Theorem in Lagrangian Dynamics}, tenth Preliminary Version.

\bibitem{GH} M.~J. Greenberg  \& J.~R. Harper.
{\it Algebraic topology.
A first course}. Mathematics Lecture Note Series, 58. Benjamin/Cummings Publishing Co., Inc., Advanced Book Program, Reading, Mass., 1981. xi+311 pp.

 \bibitem{G} M.~Gromov. {\it Partial differential relations}, volume 9 of Ergebnisse der Mathematik und ihrer Grenzgebiete (3)[Results in Mathematics and Related Areas (3)]. (1986).

 \bibitem{HLS} V.~Humili\`ere, R.~Leclercq, and S.~Seyfaddini. {\it Coisotropic rigidity and $ C^{0} $-symplectic geometry.} Duke Mathematical Journal 164.4 (2015): 767-799.

\bibitem{Her1}  M.~R.~Herman,   {\it Sur la conjugaison diffŽrentiable des diffŽomorphismes du cercle \`a des rotations}. (French) Inst. Hautes \'Etudes Sci. Publ. Math. No. 49 (1979), 5--233. 

\bibitem{Her2} M.~R. Herman,   {\it In\'egalit\'es a priori pour des tores lagrangiens invariants par des diff\'eomorphismes symplectiques.} Publications Math\'ematiques de l'IH\'ES 70 (1989): 47-101.

 
 \bibitem{KrBo} N.~Kryloff \& N.~Bogoliouboff,  {\it La th\'eorie g\'en\'erale de la mesure dans son application \`a l'\'etude des syst\`emes dynamiques de la m\'ecanique non lin\'eaire}. (French) Ann. of Math. (2) 38 (1937), no. 1, 65--113. 
 
 \bibitem{Masor} D.~Massart  \& and A.~Sorrentino, 
{\it Differentiability of {M}ather's average action and integrability on
  closed surfaces}.
 Nonlinearity, 24(6) (2011) p1777--1793.

 
 \bibitem{M} J.~N.~Mather, {\it Action minimizing invariant measures for positive definite Lagrangian systems.} Mathematische Zeitschrift 207.1 (1991): 169-207.
 
\bibitem{MS} D.~McDuff and D.~Salamon. {\it Introduction to symplectic topology}. Oxford University Press, 1998. 
 
 \bibitem{OhMu} Y.-G.~Oh \& S.~M\" uller
{\it The group of Hamiltonian homeomorphisms and $C^0$-symplectic topology. }
J. Symplectic Geom. 5 (2007), no. 2, 167--219. 

\bibitem{PR} Polterovich, Leonid, Daniel Rosen. {\it Function theory on symplectic manifolds}. American Mathematical Society, 2014.

\bibitem{R} R.~T.~Rockafellar. {\it Convex Analysis}. Princeton landmarks in mathematics. (1997).

 \bibitem{Wei1}   A.~Weinstein, {\it Lectures on symplectic manifolds}. Corrected reprint. CBMS Regional Conference Series in Mathematics, 29. American Mathematical Society, Providence, R.I., 1979. ii+48 pp.

 
\end{thebibliography}

\end{document}